\documentclass[a4,12pt]{amsart}
\usepackage{amssymb,amstext,amsmath,amscd,amsthm,amsfonts,enumerate,graphicx,latexsym, hyperref,cleveref,bm} 
\usepackage[T1]{fontenc}
\usepackage[usenames]{color}
\usepackage[all]{xy}

 \usepackage{comment}
 \usepackage{mathrsfs}
 \usepackage{tikz-cd}
\usepackage{relsize}
\usepackage{mathtools}

\numberwithin{equation}{section}
\newtheorem{lemma}[equation]{Lemma}
\newtheorem{lem}[equation]{Lemma}

\newtheorem{prop}[equation]{Proposition}
\newtheorem{cor}[equation]{Corollary}

\newtheorem{claim*}{Claim}
\newtheorem{thm}[equation]{Theorem}

\newtheorem{question}[equation]{Question}

\theoremstyle{definition}
\newtheorem{defn}[equation]{Definition}
\newtheorem{dfn}[equation]{Definition}

\newtheorem{example}[equation]{Example}
\newtheorem{ex}[equation]{Example}

\newtheorem{constr}[equation]{Construction}

\newtheorem{nota}[equation]{Notation}

\theoremstyle{remark}
\newtheorem{remark}[equation]{Remark}

\newtheorem{rem}[equation]{Remark}

\numberwithin{equation}{section}
\def\ann{\operatorname{ann}}

\def\depth{\operatorname{depth}}

\def\Ext{\operatorname{Ext}}

\def\Hom{\operatorname{Hom}}
\def\I{\operatorname{I}}
\def\id{\mathrm{id}}
\def\im{\operatorname{im}}
\def\jac{\operatorname{Jac}}
\def\ker{\operatorname{ker}}

\def\lhom{\operatorname{\underline{Hom}}}

\def\m{\mathfrak{m}}

\def\mod{\operatorname{mod}}

\def\n{\mathfrak{n}}

\def\p{\mathfrak{p}}

\def\soc{\operatorname{soc}}
\def\spec{\operatorname{Spec}}
\def\sing{\operatorname{Sing}}
\def\syz{\Omega}

\def\Tor{\operatorname{Tor}}
\def\Tr{\operatorname{Tr}}

\def\sh{\operatorname{Sh}}

\def\syz{\operatorname{Syz}}
\def\Syz{\operatorname{Syz}}

\def\sing{\operatorname{\mathrm{Sing}}}
\def\rk{\operatorname{rank}}
\def\a{\operatorname{\textbf{a}}}
\def\u{\operatorname{\textbf{u}}}
\def\ca{\operatorname{\mathrm{ca}}}

\def\I{\operatorname{\mathbf{I}}}
\def\itot{\operatorname{\mathbf{I}^{tot}}}
\def\istab{\operatorname{\mathbf{I}^{stab}}}

\def\C{\operatorname{\mathrm{C}}}

\def\F{F}
\def\ord{\operatorname{ord}}

\def\cond{\mathfrak{c}}

\def\nc{\newcommand}
\def\on{\operatorname}
\def\th{\on{th}}
\nc{\from}{\leftarrow}
\nc{\xra}{\xrightarrow}
\nc{\xla}{\xleftarrow}
\def\:{\colon}
\nc{\tQ}{\widetilde{Q}}
\nc{\tf}{\widetilde{f}}
\nc{\Db}{\on{D}^{\on{b}}}
\nc{\tD}{\widetilde{D}}
\nc{\coker}{\on{coker}}
\def\Z{\mathbb{Z}}
\nc{\odd}{\on{odd}}
\nc{\even}{\on{even}}
\nc{\tS}{\widetilde{S}}
\nc{\mf}{\on{mf}}
\nc{\hmf}{\on{hmf}}
\nc{\op}{\on{op}}
\nc{\pd}{\on{pd}}
\nc{\Perf}{\on{Perf}}
\nc{\st}{\on{st}}
\def\MR#1{}
\nc{\into}{\hookrightarrow}
\nc{\onto}{\twoheadrightarrow}

\nc{\Jac}{\on{Jac}}

\begin{document}
\setlength{\baselineskip}{14pt}
\title{Periodicity of ideals of minors in free resolutions}
\author{Michael K. Brown}

\author{Hailong Dao}

\author{Prashanth Sridhar}

\newcommand{\Addresses}{{
	\vskip\baselineskip
  	\footnotesize
  	\noindent \textsc{Department of Mathematics and Statistics, Auburn University} \par\nopagebreak
	\noindent \textit{E-mail addresses:} \texttt{mkb0096@auburn.edu, prashanth.sridhar0218@gmail.com}
  	\vskip\baselineskip
  	\noindent \textsc{Department of Mathematics, University of Kansas} \par\nopagebreak
	\noindent \textit{E-mail address:} \texttt{hdao@ku.edu}
}}

\begin{abstract}
We study the asymptotic behavior of the ideals of minors in minimal free resolutions over local rings. In particular, we prove that such ideals are eventually 2-periodic over complete intersections and Golod rings. We also establish general results on the stable behavior of ideals of minors in any infinite minimal free resolution. These ideals have intimate connections to  trace ideals and cohomology annihilators.  Constraints on the stable values attained by the  ideals of minors in many situations are obtained, and they can be explicitly computed in certain cases. 
\end{abstract}
\subjclass[2010]{}
\maketitle

\section{Introduction}

Ideals of minors of differentials in free resolutions play an important role in commutative algebra and algebraic geometry. Perhaps the most fundamental example of this is Buchsbaum-Eisenbud's acyclicity criterion \cite{BE}, which characterizes those finite free complexes that are resolutions in terms of the ranks and ideals of minors of their differentials. There has been a tremendous amount of work on understanding the differentials in a finite free resolution; such papers are too numerous to list here. For instance, a question of great interest is: how many matrices appearing in the resolution are linear (i.e. how many have entries consisting only of 0's or degree 1 forms). This so-called $N_p$ property has deep connections to the geometry of modules and ideals involved.  Fascinating work and open conjectures abound even in this rather special case of partially linear resolutions; we refer the reader to  \cite{farkas} for a relatively short survey. 

Over singular rings, most resolutions are {\it infinite}. There, much less is known about the differential matrices in general. Perhaps the most famous instance that we understand is Eisenbud's theory of matrix factorizations (\cite{Eisenbud_1980}), which essentially captures the phenomenon that over a local hypersurface, any minimal free resolution is eventually $2$-periodic. This theory has had striking ramifications far beyond commutative algebra. Another relevant active topic is the study of Koszul algebras, i.e. standard graded algebras over a field whose residue field has a linear free resolution. Interested readers may start with a couple of comprehensive surveys on infinite resolutions \cite{Avramov1998, McCPeeva}. 

In this paper, we study and establish strong stabilizing patterns, including periodicity, for ideals of minors in infinite 
minimal resolutions. There has been little indication that such behavior is expected. Apart from the hypersurface case mentioned above, our work is motivated by a recent result in \cite{DE}, where it is proved that over a large class of depth zero rings (technically, those with Burch index at least $2$), the $7^{\th}$ syzygy of {\it any} nonfree module contains the residue field as a direct summand. Thus, the ideals of minors at high enough steps of any infinite resolution over such a ring are given by the maximal ideal and its powers. In fact, the following question (and potential answers in several cases) was raised after extensive discussions and \verb|Macaulay2|~\cite{M2} experimentations by the second author and David Eisenbud:

\begin{question}
\label{introquestion1}
    Given  a positive integer $r$ and a finitely generated module $M$  over a local ring $R$, must there be a number $n(M)$ such that the sum of the ideals of $r\times r$ minors in $n(M)$ consecutive differential matrices of the minimal resolution of $M$ eventually becomes constant? Is there an integer only depending on $R$ that bounds all $n(M)$ from above?
 \end{question}

Note that if the ideal of minors (for a given rank) eventually become periodic, then the answer to the first part of the above question is yes. 
A small number of previous papers have studied the topic of ideals of minors in infinite resolutions, e.g. \cite{EG, EH, KOH1998671, Wang1994OnTF}, but as far as we know, none of them address periodicity or the weaker version in Question~\ref{introquestion1}.

Our first main result is that  ideals of minors in minimal free resolutions over local complete intersections are eventually 2-periodic. More precisely:

\begin{thm}[see Corollary~\ref{cor:easy} and Theorem~\ref{CI_periodicity}]
\label{intro1}
Let $S$ be a regular local ring, $I \subseteq S$ an ideal generated by a regular sequence, $M$ a finitely generated  module over the complete intersection $R =S/I$, $(F, d)$ the minimal $R$-free resolution of $M$, and $\I^r_i(M)$ the ideal generated by the $r \times r$ minors of $d_i : F_i \to F_{i-1}$.  
\begin{enumerate}
\item We have $\I^r_i(M) = \I^r_{i+2}(M)$ for all $r \ge 1$ and $i \gg 0$.
\item When $F$ is a Shamash resolution (e.g. when $M$ is the residue field of $R$), we have $\I^r_i(M) = \I^r_{i+2}(M)$ for all $r \ge 1$ and $i \ge r \cdot \dim(S) + 4(r-1)+ 1$.
\end{enumerate}
\end{thm}

We review the notion of a Shamash resolution below; see \S\ref{sec:shamashbackground}. Our proof of Part~(1) invokes Eisenbud-Peeva's theory of higher matrix factorizations over complete intersections \cite{EP,EP_LR}. Our argument for Part (2) relies on an interpretation of Shamash resolutions in terms of matrix factorizations, based on insights of Burke-Stevenson~\cite{BS}, Burke-Walker~\cite{BW}, and Martin~\cite{martin}. Taking $r = 1$ in Part (2), one sees that the ideals of $1 \times 1$ minors in $F$ become 2-periodic after $\dim(S) + 1$ steps, which is precisely Eisenbud's upper bound for the 2-periodicity of minimal free resolutions over hypersurfaces \cite[Theorem 6.1(i)]{Eisenbud_1980}. Since every minimal free resolution over a hypersurface is a Shamash resolution, one may therefore view Theorem~\ref{CI_periodicity}(2) as an extension of Eisenbud's periodicity theorem for minimal free resolutions over hypersurfaces. 

The dependence of the upper bound in Theorem~\ref{CI_periodicity}(2) on the rank $r$ of the minors is necessary (Remark~\ref{rem:necessary}). Moreover, while the bound in Theorem~\ref{CI_periodicity}(2) for Shamash resolutions is independent of the module $M$, no such global bound exists for general minimal free resolutions over complete intersections (Example~\ref{uniform_bound_example}). 

We now turn to Golod rings. There is a sense in which the homological behavior of complete intersections and Golod rings could not be less alike: Betti numbers over complete intersections grow as slowly as possible, while Betti numbers over Golod rings grow as fast as possible. And yet, minimal free resolutions over Golod rings also have  eventually 2-periodic ideals of minors and hence exhibit a finiteness property analogous to the complete intersection case:

\begin{thm}[see Theorem~\ref{Golod_periodicity}]
\label{intro2}
Let $S$ be a regular local ring, $I \subseteq S$ an ideal such that $R = S/I$ is Golod, $M$ a finitely generated $R$-module, $(F, d)$ the minimal $R$-free resolution of $M$, and $\I^r_i(M)$ the ideal generated by the $r \times r$ minors of $d_i : F_i \to F_{i-1}$.
\begin{enumerate}
\item We have $\I^r_i(M) = \I^r_{i+2}(M)$ for all $r \ge 1$ and $i \gg 0$.
\item If $\pd_S(R) \ge 2$, then $\I^r_i(M) = \I^r_{i+1}(M)$ for all $r \ge 1$ and $i \gg 0$.
\end{enumerate}
\end{thm}

Our proof of Theorem~\ref{intro2} uses Burke's structure theorem for minimal free resolutions over Golod rings in terms of $A_\infty$-operations and the bar construction \cite{Golod_burke} (see also \cite{iyengarchange}). The proofs of Theorems~\ref{intro1} and~\ref{intro2} both exploit multiplicative structures on free resolutions: we leverage the Eisenbud operators on the Shamash construction in the complete intersection case and $A_\infty$-operations in the Golod case.

Golod rings are quite ubiquitous in commutative algebra. For instance, by a result of Herzog-Welker-Yassemi \cite[Theorem 4.1]{HWY}, if $R$ is regular local and $I$ is an ideal in $R$, then $R/I^k$ is Golod for $k \gg 0$. Additionally, by work of Herzog-Huneke \cite[Theorem 2.1(d)]{HH}, if $R$ is an $\mathbb{N}$-graded polynomial ring over a characteristic 0 field, and $I$ is an ideal in $R$, then $R/I^k$ and $R/I^{(k)}$ are Golod for $k \ge 2$; here, $I^{(k)}$ denotes the $k^{\th}$ symbolic power of $I$. Combining Theorems~\ref{intro1} and~\ref{intro2} therefore gives a large family of local rings whose minimal free resolutions have eventually 2-periodic ideals of minors. 

For instance, given a local ring $R$ with embedding dimension $e$ and depth $d$, it is well-known that, when the codepth $e - d$ of $R$ is at most $2$, $R$ is either Golod or a complete intersection. Theorems~\ref{intro1} and~\ref{intro2} therefore immediately imply:

\begin{cor}
If $R$ is a local ring with codepth at most 2, then the ideals of minors in the minimal free resolution of any finitely generated $R$-module are eventually 2-periodic.
\end{cor}

It is not the case that every minimal free resolution has eventually 2-periodic ideals of minors. Indeed, there exists an Artinian local ring $R$ of codepth 4 and a minimal $R$-free resolution whose ideals of $1 \times 1$ minors are not $n$-periodic for any $n$: see \Cref{prop:counter}(2), which is essentially due to Gasharov-Peeva \cite[Proposition 3.1]{Gasharov-Peeva_90}. We pose the following questions on this theme:

\begin{question}
Does there exist a local ring $R$ of codepth 3  and a finitely generated $R$-module $M$ whose minimal free resolution does not have eventually $n$-periodic ideals of minors, for any $n \ge 1$? 
\end{question}

\begin{question}
\label{question2}
Does there exist a local ring $R$ and $n>2$ such that the ideals of minors of minimal free resolutions of finitely generated $R$-modules are always eventually $n$-periodic, but not 2-periodic?
\end{question}

As an application of Theorems~\ref{intro1} and~\ref{intro2}, we obtain the following periodicity result for top exterior powers of syzygies over complete intersections and Golod rings:

\begin{cor}[see Corollary~\ref{cor:exterior}]
\label{intro4}
    Let $S$ be a regular local ring, $I$ an ideal of $S$, $M$ a finitely generated $R = S/I$-module, and $\Syz_i(M)$ (resp. $\beta_i$) the $i^{\th}$ syzygy (resp. Betti number) of $M$ over $R$. 
    \begin{enumerate}
        \item If $R$ is a complete intersection, then $\bigwedge^{\beta_{i}}\syz_i(M)\cong \bigwedge^{\beta_{i+2}}\syz_{i+2}(M)$ for $i \gg 0$. If, in addition, the minimal free resolution of $M$ is a Shamash resolution, then this isomorphism holds for $i \ge \dim(S) + 1$. 
        \item If $R$ is Golod, then $\bigwedge^{\beta_{i}}\syz_i(M)\cong \bigwedge^{\beta_{i+2}}\syz_{i+2}(M)$ for $i \gg 0$. If, in addition, we have $\pd_S(R) \ge 2$, then $\bigwedge^{\beta_{i}}\syz_i(M)\cong \bigwedge^{\beta_{i+1}}\syz_{i+1}(M)$ for $i \gg 0$.
        \item Suppose $\depth(R)=0$ and that the Burch index of $R$ (see \cite[Definition 1.2]{DE}) is at least two. For all $i\geq 5$, we have $\bigwedge^{\beta_i}\syz_i(M)\cong k$.
    \end{enumerate}
\end{cor}

When $R$ is a hypersurface, Corollary~\ref{intro4}(1) follows from \cite[Theorem 6.1(i)]{Eisenbud_1980}. We refer the reader to \cite{DE} for various concrete examples of rings of depth zero and Burch index at least two. 

\medskip
The rest of the main results of this paper concern \emph{stable ideals of minors} (see Definition~\ref{def:stable}). Given a module $M$ with minimal free resolution $F$, the stable ideal of $r \times r$ minors of $M$ is the limit as $n \to \infty$ of the sum of all the ideals of $r \times r$ minors in the smart truncation $F_{\ge n}$. Our key results in this direction are summarized as follows:

\begin{thm}
[see Theorems~\ref{thm:nonzero}, \ref{thm:jacobian}, \ref{residue_field_stable_ideals}, and~\ref{thm:conductor}]
\label{intro5}
Let $(R, \m, k)$ be a local ring and $M$ a finitely generated $R$-module. Denote by $\istab(M, r)$ the stable ideal of $r \times r$-minors in the minimal free resolution of $M$ (see Definition~\ref{def:stable}). We have the following:
\begin{enumerate}
\item $\istab(M, 1) \ne 0$.
\item Let $\Jac(R)$ be the Jacobian ideal of $R$ (see \cite[\S 3]{IYENGAR2021280} for the definition of $\Jac(R)$). There exists  $t \ge 1$, which is independent of $M$, such that $\Jac(R)^t \subseteq \istab(M, 1)$.
\item Let $K$ be the Koszul complex on a minimal generating set of $\m$. If $\dim_k H_1(K)\geq 2$, then $\istab(k,r)=\m^r$ for all $r\geq 1$. 
\item Assume $R$ is a 1-dimensional analytically unramified local ring, and let $\cond_R$ denote the conductor ideal of $R$ (see Definition \ref{def:conductor}). 
\begin{enumerate}
\item We have $\cond_R\subseteq \istab(M,1)$.
\item Assume $R$ is Gorenstein with infinite residue field. There exists a finitely generated $R$-module $N$ such that $\cond_R =  \istab(N,1)$  if and only if $R$ is regular or a hypersurface of multiplicity 2.
\end{enumerate}
\end{enumerate}
\end{thm}

Concerning Theorem~\ref{intro5}(2), we note that Iyengar-Takahashi's definition of the Jacobian ideal of a local ring generalizes the usual notion for affine algebras over a field. Theorem~\ref{intro5}(2) is an instance of the philosophy that the singularities of the ring $R$ govern the asymptotic behavior of minimal $R$-free resolutions \cite[pp. 4]{Avramov1998}. Parts (2) and (3) of Theorem~\ref{intro5} build on closely related results of Wang \cite[Theorems 5.1 and 5.2]{Wang1994OnTF} and Eisenbud-Green \cite[Theorem 1.1]{EG}, respectively. See \S\ref{pastresults} for a more detailed recollection of the history of results on ideals of minors in minimal free resolutions that influenced this work.

We now give a brief summary of the paper. In \S\ref{sec:stable}, we prove Theorem~\ref{intro5}, among many other related results on stable ideals of minors. In \S\ref{CI_section} and \S\ref{golod_section}, we prove Theorems~\ref{intro1} and~\ref{intro2}, respectively. We discuss in \S\ref{sec:counterexample} the aforementioned counterexample to periodicity of ideals of minors over general local rings. Finally, in \S\ref{sec:moreresults}, we explore some applications of Theorems~\ref{intro1} and~\ref{intro2}, e.g. Corollary~\ref{intro4}.
   
\subsection*{Acknowledgments}
Many questions and ideas that started this work were raised during the second author's collaboration with David Eisenbud, whose insights and generosity we gratefully acknowledge. We thank Luigi Ferraro, Srikanth Iyengar, Frank Moore, Claudiu Raicu, and Mark Walker for valuable conversations. We also thank Sarasij Maitra and Josh Pollitz for catching an error in an earlier draft of this paper. The computer algebra system \verb|Macaulay2|~\cite{M2} provided valuable assistance throughout this project. The second author was partially supported by Simons Collaboration Grant FND0077558. The third-named author acknowledges the grant GA \v{C}R 20-13778S from the Czech Science Foundation for support during a part of this work. Finally, we thank the referee for their helpful comments.
\section*{Notation and Conventions}
\label{intronotation}
\begin{enumerate}
\item All rings are assumed to be commutative, and \emph{local} means \emph{Noetherian local}.
\item We use homological indexing throughout; the shift of a complex $C$ is given by $C[1]_j = C_{j + 1}$ and $d_{C[1]} = -d_C$. 
\item For a complex $C$, we let $\inf(C) = \inf\{i \;|\; C_i\neq 0\}$, and similarly for $\sup(C)$. 
\item Given a ring $R$ and an $R$-linear map $g : F \to F'$ of free $R$-modules of finite rank, we denote by $\I^r(g)$ the ideal in $R$ generated by the $r \times r$ minors of $g$. When $r$ is greater than the rank of $F$ or $F'$, we set $\I^r(g) = 0$; and when either $F$ or $F'$ is 0, we set $\I^r(g) = R$. 
\item Given a complex $(C, d)$ of finitely generated free $R$-modules, let $\I^r_i(C)$ denote the ideal $\I^r(d_i : C_i \to C_{i-1})$. When $R$ is local and $M$ is a finitely generated $R$-module with minimal free resolution $F$, we set $\I^r_i(M) \coloneqq \I^r_i(F)$; these ideals are independent of the choice of $F$. 
\item Given a local ring $R$ and a finitely generated $R$-module $M$, we denote by $\syz_n(M)$ the $n^{\th}$ syzygy in a minimal free resolution of $M$.
\item Although we work over a local ring throughout most of the paper, all of our results that are stated over local rings hold over an $\mathbb{N}$-graded ring $R$ such that $R_0$ is local and $R$ is a finitely generated $R_0$-algebra.
\end{enumerate}

\section{Stable Ideals of Minors}
\label{sec:stable}

\subsection{Past Results}
\label{pastresults}

Perhaps the first systematic study of ideals of minors in free resolutions was undertaken by Eisenbud-Green in \cite{EG}. Among other results, Eisenbud-Green prove in \cite{EG} a conjecture of Huneke regarding the relationship between ideals of minors in a free resolution of a module $M$ and the annihilator of $M$. Building on \cite{EG} and also results of Popescu-Roczen \cite{Popescu-Roczen}, Wang studies in \cite{Wang1994OnTF} the connection between the aforementioned conjecture of Huneke and the annihilator ideal of the functor $\Ext^{d+1}_R( - ,- )$, where $R$ is a local ring of Krull dimension $d$. The main results of \cite{Wang1994OnTF} include versions of Eisenbud-Green's main result \cite[Theorem 1.1]{EG} with the annihilator of the module replaced by the Jacobian ideal \cite[Theorems 5.1 and 5.2]{Wang1994OnTF}. Some results of Wang in \cite{Wang1994OnTF} are later built upon by Iyengar-Takahashi in their study of cohomology annihilators \cite{iyengar_takahashi_ca}. Koh-Lee also extend some results of \cite{EG} and \cite{Wang1994OnTF} in their paper \cite{KOH1998671}; we revisit Koh-Lee's results in detail in \S\ref{kohlee}.

Recent work with a specific focus on the behavior of the entries of matrices in minimal free resolutions includes the following. Motivated by a question of Eisenbud-Reeves-Totaro in \cite{ERT} about bounding the degrees of the entries in minimal graded free resolutions, Eisenbud-Huneke establish in \cite{EH} a uniform Artin-Rees-type property for the maps in the minimal free resolution of a finitely generated module over a local ring with an isolated singularity. Additionally, motivated by work of the second author and Kobayashi-Takahashi \cite{DKT}, the second author and Eisenbud obtain in \cite{DE} constraints on the ideals of $1 \times 1$ minors in minimal free resolutions over local rings via their notion of the Burch index.

Our main goal in this section is to prove each part of Theorem~\ref{intro5}, building on much of the above work. We provide detailed commentary throughout this section on the relationship between the present work and the results discussed above.

\subsection{Definitions}

While our ultimate interest is in minimal free resolutions of modules, many of our results easily extend to minimal free resolutions of (bounded below) complexes, and so we work at this level of generality throughout this section.

\begin{dfn}
\label{def:stable}
Let $R$ be a local ring and $\C^{+}(\mod \- R)$ the category of bounded below complexes of finitely generated $R$-modules. Suppose $M\in \C^{+}(\mod \- R)$, and let $\F\xrightarrow{\simeq}M$ be its minimal free resolution (such a resolution exists and is unique; see e.g. \cite[Proposition 4.4.2]{roberts_1998}). For any positive integer $r$, define
\begin{enumerate}
    \item $\itot(M,r) = \sum_{n\geq \inf(F)}\I^r_n(\F)$, the \emph{total ideal of $r \times r$ minors of $M$.}
    \item $\istab(M,r)=\bigcap_{n\geq \inf(F)}\itot(\F_{\geq n},r)$, the \emph{stable ideal of $r \times r$ minors of $M$.} Here, $\F_{\geq n}$ denotes the smart truncation.
    \item $\istab(\C^{+}(\mod \- R),r)=\bigcap_{M\in \C^{+}(\mod \- R)} \itot(M,r)$.
\end{enumerate}
 Let $\C^{b}(\mod \- R)$ and $\C^{+, f}(\mod \- R)$ be the full subcategories of $\C^{+}(\mod \- R)$ given by bounded complexes and complexes with bounded homology, respectively. We define the ideals $\istab(\C^{b}(\mod \- R),r), \istab(\C^{+, f}(\mod \- R),r)$, and $\istab(\mod(R),r)$ analogously.
\end{dfn}
Let $M \in \C^+(\mod \- R)$ with minimal free resolution $F$.  The \emph{projective dimension} of $M$ is defined to be $\sup(F)$. The complex $M$ is called \emph{perfect} if it has finite projective dimension. 
The following results are immediate consequences of the definitions:
\begin{prop}
Let $R$ be a local ring and $M \in \C^+(\mod \- R)$.
\begin{enumerate}
\item We have $\itot(M,1) \neq 0$.
\item The complex $M$ is perfect if and only if $\istab(M,r)=R$ for some $r>0$. 
\item $\istab(\C^+(\mod \- R),r)=\bigcap_{M\in \C^+(\mod \- R)} \istab(M,r)$.

\item $R$ has finite global dimension if and and only if $\istab(\C^{b}(\mod \- R),r)=R$.

\end{enumerate}
\end{prop}

\begin{remark}
It is possible that $\itot(M,r)=0$ for $r>1$. For instance, take $R$ to be Artinian and $r$ greater than the Loewy length of $R$. 
\end{remark}
\subsection{Extending some results of Koh-Lee} 
\label{kohlee}
Our main results on stable ideals of minors rely on work of Koh-Lee \cite{KOH1998671} on ideals of minors in minimal free resolutions of modules. Since we work with minimal free resolutions of complexes rather than modules, we must slightly extend these results. The arguments we give here are minor variations of those appearing in \cite{KOH1998671}.

The socle of a complex $M$ is defined to be the complex $\lhom_R(k,M)$, where $k$ is concentrated in degree zero. Given a complex $(M, d^M)$, set 
\[
s(M)\coloneqq \inf\{t\geq 1\;|\; \soc(M)\not\subseteq \m^tM+\im(d^M)\};
\]
this is a generalization of \cite[Definition 1.1]{KOH1998671}.

\begin{prop}[cf. \cite{KOH1998671} Proposition 1.2(i)] \label{koh_lee_prop}
    Let $(R,\m)$ be a  local ring, $(M, d^M)\in\C^+(\mod \- R)$, and $(N,d^N)\in \C^{b}(\mod \- R)$. Assume $N$ is minimal and $\soc(N)\not\subseteq \im(d^N)$. Let $\F\xrightarrow{\simeq}M$ be the minimal free resolution of $M$. If an integer $i$ satisfies
    \begin{enumerate}
    \item[(a)] $\inf(F) + \sup(N) \le i\le \sup(F)+\inf(N)$, and 
    \item[(b)] $\Tor^R_i(M,N)=0$;
    \end{enumerate}
    then $\I^1_{i+1}(\F)\not\subseteq \m^{s(N)}$.
\end{prop}

\begin{remark}
\label{rem:socle}
\text{ }
\begin{enumerate}
\item If $\soc(M)\subseteq \im(d^M)$, then $s(M)=\infty$. If $M \in \C^b(\mod \- R)$, and $\soc(M)\nsubseteq \im(d^M)$, then $s(M) < \infty$. In particular, the complex $N$ in the statement of Proposition~\ref{koh_lee_prop} satisfies $s(N) < \infty$.
\item When the complex $N$ in Proposition~\ref{koh_lee_prop} satisfies $\inf(N) \ge -1$, the assumption $i \le \sup(F) + \inf(N)$ is superfluous, as $\I^1_{i+1}(F) = R$ for all $i \ge \sup(F)$.
\end{enumerate}
\end{remark}

\begin{proof}[Proof of Proposition~\ref{koh_lee_prop}]
    Set $G \coloneqq\F\otimes_R N$.  
    Since we have $\Tor^R_i(M,N)=0$, the sequence 
    $
    G_{i-1} \xla{} G_i \xla{} G_{i+1}
    $ is exact in the middle. Suppose $j$ is an integer such that $N_j \ne 0$. The inequalities
$
\inf(N) \le j \le \sup(N)
$
imply
$
i - \sup(N) \le i - j \le i - \inf(N).
$
Applying our assumption (a), we have
$
\inf(F) \le i - j \le \sup(F),
$
i.e. $F_{i - j} \ne 0$. Thus, setting $r_k = \on{rank} F_k$, we conclude that each term of the sum
$$
G_i = \bigoplus_{j = \inf(N)}^{\sup(N)}  N_j^{r_{i - j}}
$$
is nonzero. Since $\F$ and $N$ are minimal, we have 
$$
\soc(N_j)^{r_{i-j}}\subseteq \ker(d^G_i)=\im(d^G_{i+1}) \subseteq \I^1_{i+1}(\F)N_j^{r_{i-j}}+\im(d^N_{j+1})^{r_{i-j}}
$$
for all $\inf(N) \le j \le \sup(N)$. It follows that $\soc(N) \subseteq \I^1_{i+1}(F)N + \im(d^N)$.
By the definition of the integer $s(N)$, we therefore have $\I^1_{i+1}(\F)\not\subseteq \m^{s(N)}$.
\end{proof}

\begin{prop}[cf. \cite{KOH1998671} Theorem 1.7(i)]\label{complexes_koh_lee}
    Let $(R,\m)$ be a  local ring. There exists an integer $k>0$ such that, given any $M\in \C^{b}(\mod \- R)$ with minimal free resolution $\F\xrightarrow{\simeq}M$, we have  $\I^1_{i+1}(\F)\not\subseteq \m^k$ for all $i>\sup(M)+\depth(R)$. 
\end{prop}
\begin{proof}
    Let $(\underline{x})\subseteq R$ be a maximal regular sequence, and set $N\coloneqq R/(\underline{x})$; if $\depth{R}=0$, take $N=R$. 
    Fix $i > \sup(M) + \depth(R)$. Since $\pd_R(N)=\depth(R)$, we have $\Tor^R_i(M,N)=0$. Certainly $\inf(F) \le \sup(M) < i$; it therefore follows from Proposition~\ref{koh_lee_prop} and Remark~\ref{rem:socle}(2) that $\I^1_{i+1}(\F)\not\subseteq \m^k$ for $k=s(N)$. 
\end{proof}

\begin{rem}\label{koh-lee_rem}
The conclusion of \Cref{complexes_koh_lee} cannot hold in general for a smaller homological index. Indeed, consider Koszul complexes on maximal regular sequences of elements contained in arbitrarily high powers of the maximal ideal. 

\end{rem}

\subsection{Constraints on stable ideals of minors}
\begin{thm}
\label{thm:nonzero}
    Let $(R,\m)$ be a  local ring. If $M$ is a bounded complex of finitely generated $R$-modules, then $\istab(M,1)\neq 0$.
\end{thm}
\begin{proof}
    Let $\F$ be the minimal free resolution of $M$, and set $J_j\coloneqq\itot(\F_{\geq j},1)$. Recall that $\istab(M,1) \coloneqq \bigcap_{j = \inf(F)}^{\infty} J_j$. We first assume that $R$ is $\m$-adically complete. Suppose $\istab(M,1) = 0$. By \cite[Lemma 7]{Chevalley}, for every integer $t\geq 1$, there exists an integer $n(t)$ such that $J_{n(t)}\subseteq \m^t$. By \Cref{complexes_koh_lee}, there exists an integer $k$ such that, for all $i>\sup(M)+\depth(R)+1$, we have $\I^1_i(\F)\not\subseteq \m^k$. In particular, $J_j \nsubseteq \m^k$ for $j > \sup(M)+\depth(R)+1$. Choosing $j > \max\{\sup(M)+\depth(A)+1,n(k)\}$ therefore yields $J_i \subseteq \m^k$ and $J_i \nsubseteq \m^k$, a contradiction. 
    
    As for the general case: we have $\widehat{\istab(M,1)} = \bigcap_{j = \inf(F)}^{\infty} \widehat{J_j} = \istab(\widehat{M}, 1)$, where the first equality follows since completion commutes with intersections, and the second is due to the flatness of completion. Since $\istab(\widehat{M}, 1) \ne 0$, we must have ${\istab(M,1)} \ne 0$.
\end{proof}

Let $R$ be a local ring. For all $n \ge 0$, we let
$$
\ca^n(R) = \bigcap_{M,N\in \mod(R)}\ann_R(\Ext^{\ge n}_R(M,N))
$$
Following Iyengar-Takahashi \cite[Definition 2.1]{iyengar_takahashi_ca}, the \emph{cohomology annihilator of $R$} is defined to be  $\ca(R) \coloneqq \bigcup_{n \ge 0} \ca^n(R)$. 
By Noetherianity, $\ca(R) = \ca^s(R)$ for some $s \gg 0$. 

Before stating the next result, we recall that $\C^{+, f}(\mod \- R)$ denotes the category of bounded below complexes of finitely generated $R$-modules with bounded homology.

\begin{prop}\label{istab_modR}
    Let $(R,\m,k)$ be a local ring. 
    \begin{enumerate}
    \item Given $M\in \C^{+, f}(\mod \- R)$, we have $\ca(R)\subseteq \istab(M,1)$. Thus, $\ca(R)\subseteq \istab(\C^{+, f}(\mod \- R),1)= \istab(\mod(R),1)$. In particular, if $\ca(R)\neq 0$, then $\istab(\C^{+, f}(\mod \- R),1)\neq 0$.
    \item If $R$ is not regular, then $\istab(k,1)=\m$, and $\istab(L,1)$ is $\m$-primary for any $R$-module $L$ of finite length. 
    \end{enumerate}
\end{prop}
\begin{proof}
    Let $\F$ be the minimal free resolution of $M$. Given $i\geq 0$, it suffices to show $\ca(R)\subseteq I \coloneqq \itot(\F_{\geq i},1)$; recall that $F_{\ge i}$ denotes the smart truncation of $F$ at $i$.  Choose $s \ge 0$ such that $\ca(R) = \bigcap_{A,B\in \mod(R)}\ann_R(\Ext^{\ge s}_R(A,B))$. Choose also $t \ge 0$ such that $H_{\ge t}(M) = 0$, so that $F_{\ge j}$ is the minimal free resolution of a module $N_j$ for all $j \ge t$. Let $r = \max\{i, s+t\}$. Observe that $\Tor^R_{r}(M,R/I) = \Tor^R_r(N_r, R/I)$ is a free $R/I$-module, and so $\ann(\Tor^R_{r}(N_r,R/I))= I$. 
    We therefore have 
    $$
\ca(R)\subseteq \ann_R(\Ext^r_R(N_r,R/I)) = \ann_R(\Tor^R_r(N_r,R/I)) = I,
$$
    where the first equality follows by \cite[Corollary 1.6(3)]{Wang1994OnTF} (and taking syzygies). This proves (1). As for (2): setting $\ca(L)\coloneqq\bigcup _{j\geq 1}(\bigcap_{N\in \mod(R)}\ann_R(\Ext^j_R(L,N))) $, we clearly have that $\ann(L)\subseteq \ca(L)$. Arguing as in the proof of (1), we conclude $\ca(L)\subseteq \istab(L,1)$. 
\end{proof}

 Given a local ring $R$, an object $M \in \C^{+, f}(\mod \- R)$ with minimal free resolution $F$, and $r \ge 1$; we say $\istab(M,r)$ is \emph{finitely determined} if it is equal to $\itot(F_{\ge k}, r)$ for some $k \gg 0$. Similarly, we say $\istab(\C^{+, f}(\mod \- R),r)$ is \emph{finitely determined} if it is equal to the intersection of finitely many ideals of the form $\istab(M_i,r)$, where each $\istab(M_i,r)$ is finitely determined. 
 The following is immediate from \Cref{istab_modR}(1):

   \begin{cor}\label{ca_R_m_primary}
       Let $(R,\m)$ be a local ring such that $\ca(R)$ is $\m$-primary. For any $M\in\C^{+, f}(\mod \- R)$, $\istab(M,1)$ is finitely determined. Moreover, $\istab(\C^{+, f}(\mod \- R),1)$ is finitely determined.
       \end{cor}

We also have:
       \begin{cor}
       \label{findet}
Suppose $R$ is a localization of a finitely generated algebra over a field or an equicharacteristic excellent local ring. If $R$ is an isolated singularity, then for any $M\in\C^{+, f}(\mod \- R)$, $\istab(M,1)$ is $\m$-primary and finitely determined. Moreover, $\istab(\C^{+, f}(\mod \- R),1)$ is $\m$-primary and finitely determined.
       \end{cor}
       \begin{proof}
           Apply \cite[Theorems 5.3 and 5.4]{iyengar_takahashi_ca}, \Cref{istab_modR}(1), and \Cref{ca_R_m_primary}.
       \end{proof}

It follows from Theorems~\ref{CI_periodicity} and~\ref{Golod_periodicity} below that, when $R$ is a complete intersection or Golod, $\istab(M,r)$ is finitely determined for any $r\geq 1$ and $M \in \C^{+, f}(\mod \- R)$. We ask:

\begin{question}
    Let $R$ be a Noetherian local ring and $M$ a finitely generated $R$-module. Is $\istab(M,r)$ finitely determined for any $r\geq 1$?
\end{question}

\begin{thm}
\label{thm:jacobian}
    Let $R$ be a local ring, $\sing(R)$ the singular locus of $R$, and $\jac(R)$ the Jacobian ideal of $R$; see \cite[\S 3]{IYENGAR2021280} for the definition of $\jac(R)$. We have:

    \begin{enumerate}
    \item $\jac(R)^t \subseteq \istab(\C^{+, f}(\mod \- R),1)$ for some $t \ge 1$.
   \item If $R$ is a localization of a finitely generated algebra over a field, or if $R$ is an equicharacteristic excellent local ring, then $V(\istab(\C^{+, f}(\mod \- R),1))\subseteq \sing(R)$. Thus, if $\sing(R) \ne \spec(R)$ (e.g. if $R$ is reduced), then $\istab(\C^{+, f}(\mod \- R),1)\neq 0$.
   \item Suppose $R$ is complete, Cohen-Macaulay and contains a field. We have $\mathfrak{R}^R\subseteq \istab(\C^{+, f}(\mod \- R),1)$, where $\mathfrak{R}^R$ is the Noether different of $R$; see \cite[Definition 5.5]{Wang1994OnTF}.
    \end{enumerate}

\end{thm}
\begin{proof}
  By \cite[Theorem 3.4]{IYENGAR2021280}, some power of $\jac(R)$ is contained in $\ca(R)$. Part (1) therefore follows from \Cref{istab_modR}(1). For (2), \cite[Theorems 5.3 and 5.4]{iyengar_takahashi_ca} imply that $V(\ca(R))=\sing(R)$ under the given hypotheses. The claims therefore follow from \Cref{istab_modR}(1). For (3), by \cite[Corollary 5.13]{Wang1994OnTF}, $\mathfrak{R}^R$ annihilates the functor $\Ext^1_R(M, - )$ for any maximal Cohen-Macaulay module $M$. It follows that $\mathfrak{R}^R\subseteq \ca(R)$ and we conclude by \Cref{istab_modR}(1).
\end{proof}

\begin{rem}
    Under extra hypotheses, \cite[Theorem 5.2]{Wang1994OnTF} implies \Cref{thm:jacobian}(1).
\end{rem}

For $r > 1$, the ideals $\istab(M,r)$ can be 0. For example, if the Betti numbers of the minimal free resolution of $M \in \C^{+, f}(\mod \- R)$ are bounded, then $\itot(M,r)=\istab(M,r)=0$ for $r \gg 0$. Such is the case for any finitely generated module $M$ over a  hypersurface ring, although the choice of $r$ is dependent on $M$.  
Moreover, even if the Betti numbers of $M$ are unbounded, $\istab(M,r)$ could vanish for other reasons. For instance, if $R$ is an Artinian ring, then 
$\itot(M,r)=\istab(M,r)=0$ for $r>\ell\ell(R)$, the Loewy length of $R$. However, we have the following result on ideals of higher rank minors of the residue field:

\begin{thm}\label{residue_field_stable_ideals}
    Let $(R,\m,k)$ be a local ring and $K$ the Koszul complex on a minimal generating set of $\m$.
    \begin{enumerate}
    \item If $\dim_k H_1(K)\geq 2$, then $\istab(k,r)=\m^r$ for all $r\geq 1$. 
    \item If $R = S/I$, where $(S,\n)$ is a regular local ring and $I\subseteq \n^2$, then $\istab(k,r)=\m^r$ for all $r\geq 1$ unless $R$ is a hypersurface.
    \end{enumerate}
\end{thm}
\begin{proof}
   Let $(F, d)$ be the free resolution of $k$ constructed in  \cite{Tate_57} (which is minimal by a result of Gulliksen \cite{gulliksen}). The resolution $F$ is a dg $R$-algebra: as an $R$-algebra, it is a tensor product of a divided power algebra in some collection of even (and nonnegative) degree variables and an exterior algebra in some collection of odd (and positive) degree variables. Equip $F$ with an $R$-basis given by monomials in these variables. Denote the exterior variables comprising our basis of $F_1$ by $X_{1}, \dots, X_{n}$, and set $x_\ell \coloneqq d(X_{\ell})$; the elements $x_1, \dots, x_n$ form a minimal generating set for $\m$. 
   Since $\dim_k H_1(K)\geq 2$, we may also choose two distinct divided power variables $T_{1},T_{2}$ in our basis of $F_2$.
   
   Fix $r \ge 1$, and let $x_{j_1}, \dots, x_{j_r}$ be a set of $r$ (not necessarily distinct) elements in $\{x_1, \dots, x_n\}$. To prove (1), it suffices to show that, given an  integer $N \ge 1$, there is some $M \ge N$ such that the matrix $d_M : F_M \to F_{M-1}$ (with respect to our chosen basis of $F$) has an $r \times r$ diagonal submatrix with the elements $x_{j_1}, \dots, x_{j_r}$ on the diagonal. Notice that, for any $k_1, k_2 \ge 0$ and $1 \le \ell \le n$, we have
   \begin{equation}
   \label{eq:relation}
   d(X_{\ell} T_{1}^{k_1} T_{2}^{k_2}) = x_{\ell} T_{1}^{k_1} T_{2}^{k_2} + (\text{terms with $X_{\ell}$ as a factor)}.
   \end{equation}
   Without loss of generality, assume $N$ is odd. Set $M \coloneqq \max\{N, 2(r-1) + 1\}$ and $s   \coloneqq \frac{M - 1}{2}$. The relation \eqref{eq:relation} implies that the following is a submatrix of $d_M$:
 $$
    \bordermatrix{~ & X_{j_1}T_1^{s} & X_{j_2}T_1^{s-1}T_2 & \cdots & X_{j_{r-1}}T_1^{s-(r-2)}T_2^{r-2}& X_{j_r}T_1^{s-(r-1)}T_2^{r-1}\cr
                  T_1^{s} & x_{j_1} & 0 & \cdots & 0 & 0\cr
                  T_1^{s-1}T_2 & 0 & x_{j_2} & \cdots & 0& 0\cr
                  \vdots & \vdots & \vdots & \cdots & \vdots & \vdots \cr
                  T_1^{s-(r-2)}T_2^{r-2} & 0 & 0 & \cdots & x_{j_{r-1}}& 0\cr
                  T_1^{s-(r-1)}T_2^{r-1} & 0 & 0 & \cdots & 0& x_{j_r}\cr
                  }.
$$  
This proves (1). To prove (2), we observe that $\dim_k H_1(K)$ is equal to $\dim_k(I/\n I)$, the minimal number of generators of $I$ (see e.g. \cite[Lemma 1.4.15]{Gulliksen1969HomologyOL}). Thus, $R$ is a hypersurface if and only if $\dim_k H_1(K)=1$. Now apply (1).
   \end{proof}

 \begin{rem}
When $\mathbb{Q} \subseteq R$ and $\depth(R) > 0$, one may use \cite[Theorem 1]{EG} to prove \Cref{residue_field_stable_ideals}.
 \end{rem}

\subsection{Stable ideals of minors and the conductor}

\begin{defn}
\label{def:conductor}
Let $R$ be a reduced ring. The \emph{conductor} of $R$ is defined to be
$
\cond_R = \{r \in R \text{ : } r \cdot \overline{R} \subseteq R\},
$
where $\overline{R}$ denotes the integral closure of $R$.
\end{defn}
The main goal of this subsection is to prove Theorem~\ref{thm:conductor}, which illustrates a relationship between stable ideals of $1 \times 1$ minors and the conductor in the case of a 1-dimensional analytically unramified local ring (cf. \cite[Corollary 3.3]{Wang1994OnTF}). We first recall some background on trace ideals:

\begin{defn}
Given a module $M$ over a commutative ring $R$, the \emph{trace ideal} of $M$ over $R$, denoted $\Tr_R(M)$, is the ideal in $R$ generated by the images of all morphisms $f \in M^\vee \coloneqq \Hom_R(M, R)$. Alternatively, $\Tr_R(M)$ is the image of the evaluation map $M \otimes_R M^\vee \to R$. 
\end{defn}

\begin{remark}
\label{rem:ref}
If $M$ is reflexive (i.e. the natural map $M \to (M^\vee)^\vee$ is an isomorphism), then $\Tr_R(M) = \Tr_R(M^\vee)$.  
\end{remark}

The connection between trace ideals and ideals of $1 \times 1$ minors is provided by the following result:

\begin{prop}\label{syzygy_trace}
Let $R$ be a Gorenstein local ring and $M$ a finitely generated $R$-module whose minimal free resolution $F$ is infinite. 
\begin{enumerate}
\item If 
$n > \dim(R) + 1$, then $\I^1_n(\F)=\Tr_R(\syz_n(M))$. \item If $M$ is maximal Cohen-Macaulay (MCM), then $\I^1_1(M)= \Tr_R(\syz_1(M))$.
\end{enumerate}
\end{prop}
\begin{proof}
Let $i \in \{n-2, n-1\}$. Observe that $\Ext^1_R(\syz_{i}(M), R) = \Ext^{i + 1}_R(M, R) = 0$, since $R$ is Gorenstein and $i +1 \ge  n-1 > \dim(R)$. Dualizing the exact sequence
$
0 \to \syz_{i+1}(M) \to \F_{i} \to \syz_{i}(M) \to 0
$
therefore gives the exact sequence
$
0 \to \syz_{i}(M)^\vee \to \F_{i}^\vee \to  \syz_{i+1}(M)^\vee \to  0.
$
Combining these exact sequences for $i = n-2$ and $i = n-1$ yields a free presentation $F_{n-2}^\vee \to F_{n-1}^\vee \to \Syz_n(M)^\vee \to 0$. It thus follows from \cite[\S 2.1.3]{Trace_Lindo} (see also \cite[Remark 3.3]{Vasconcelos1991}) that $\Tr_R(\syz_n(M)^\vee)=\I^1_n(\F)$. Since $R$ is Gorenstein and $n > \dim(R)$, $\Syz_n(M)$ is MCM and therefore reflexive.
By Remark~\ref{rem:ref}, we conclude that $\Tr_R(\syz_n(M)) = \Tr_R(\syz_n(M)^\vee) = \I^1_n(\F)$; this proves (1). For (2): we may assume without loss of generality that $M$ has no free summand. (2) therefore follows from (1), along with the observation that maximal Cohen-Macaualy modules with no free summands over Gorenstein rings are infinite syzygies.
\end{proof}

\begin{cor}\label{Gorenstein_trace}
Let $(R,\m)$ be a Gorenstein local ring and $M$ a maximal Cohen-Macaulay $R$-module. 
\begin{enumerate}
\item The ideal $\itot(M,1)$ is a trace ideal. 
\item If $\istab(M,1)$ is $\m$-primary, then it is also a trace ideal.     
\end{enumerate}
\end{cor}

\begin{proof}
Observe that $\sum_{i \in I} \Tr_R(M_i) = \Tr_R(\bigoplus_{i \in I} M_i)$ for any collection $\{M_i\}_{i \in I}$ of $R$-modules. Part (1) therefore follows from Proposition~\ref{syzygy_trace}(2). Let us now prove (2). Let $F$ be the minimal free resolution of $M$. Since $\istab(M,1)$ is $\m$-primary, it is equal to $\itot(F_{\ge n}, 1)$ for $n \gg 0$ (recall that $F_{\ge n}$ denotes the smart truncation). Now apply (1).
\end{proof}

The following fact is well-known to experts, but due to lack of an exact reference, we include a proof.
\begin{lemma}
\label{nagata}
Let $(R, \m)$ be a 1-dimensional analytically unramified local ring, and denote by $\overline{R}$ (resp. $\widehat{R}$) the integral closure (resp. $\m$-adic completion) of $R$.

\begin{enumerate}
    \item $\overline{R}$ is a finitely generated $R$-module.
    \item The integral closure of $\widehat{R}$ is isomorphic to $\overline{R} \otimes_R \widehat{R}$.
\end{enumerate}
\end{lemma}

\begin{proof}
For (1): since $R$ is reduced, we have $\overline{R}\cong \prod \overline{R/P}$, where the product ranges over the minimal primes of $R$. Moreover, for every minimal prime $P$ of $R$, $R/P$ is analytically unramified; see e.g. \cite[Proposition 9.1.3]{HS}. The conclusion of (1) then follows from \cite[Appendix 1, Proposition 1]{Nag58}. For (2), note that $\overline{R}$ is a normal semi-local ring by (1). Since $\overline{R} \otimes_R \widehat{R}$ is birational and finitely generated over $\widehat{R}$, it suffices to show $\overline{R} \otimes_R \widehat{R}$ is normal.  Since $\overline{\widehat{R}}$ is a finitely generated module over $\widehat{R}$ by (1), there exists a non-zero-divisor $c\in \overline{\widehat{R}}$ such that $c\overline{\widehat{R}}\subseteq \widehat{R}$; in particular, $\cond_{\widehat{R}}\neq 0$. Since $\cond_{\widehat{R}}$ is primary to the maximal ideal, we have $\m^n\subseteq \cond_{\widehat{R}}$ for some $n$. Thus, we can choose a non-zero-divisor $t\in \cond_{\widehat{R}}\cap R$. In particular, $t$ is a nonzero, nonunit element of $\overline{R}$. Now apply \cite[37.2]{Nagata} to the ring $\overline{R}$ and the element $t \in \overline{R}$ to conclude that $\overline{R} \otimes_R \widehat{R}$ is normal (the assumption on $tR$ in this result of Nagata is satisfied since $t$ is a non-zero-divisor and $\dim(R) = 1$).
\end{proof}

We recall that the \emph{order} of an ideal $I$ in a local ring $(R, \m)$ is $\ord(I) \coloneqq \sup\{k \text{ : } I \subseteq \m^k\}$.

 \begin{lemma}\label{Artinian_socle}
    Let $R$ be an Artinian local ring and $M\in \C^+(\mod \- R)$. We have $\soc(R)\subseteq \istab(M,1)$, and
    $\ord(\istab(M,1))\leq \lfloor l/2 \rfloor$, where $l$ is the Loewy length of $R$. In particular, $\istab(\C^+(\mod \- R),1) \neq 0$.
\end{lemma}

\begin{proof}
    Fix an integer $n\geq 2+\inf(M)$. Let $(F,d)$ be the minimal free resolution of $M$, and write $\beta_{n-1} = \on{rank} F_{n-1}$. The submodule $\soc(R)^{\oplus \beta_{n-1}}\subseteq \F_{n-1}$ is contained in $\ker(d_{n-1}) = \im(d_n)$, since $\F$ is minimal. We therefore have $\soc(R) \subseteq \I^1(v)$ for every row $v$ of $d_n$; in particular, $\soc(R) \subseteq \I^1_n(M)$. Thus, $\soc(R)\subseteq \istab(M,1)$. It follows that $\soc(R) \subseteq \istab(\C^+(\mod \- R),1)$; in particular, $\istab(\C^+(\mod \- R),1) \ne 0$.
    
    Since $R$ is Artinian, $\istab(M,1)=\itot(\F_{\geq k},1)$ for some $k\geq \inf(M)$. Fix $n>k$. If $\ord(\I^1_{n}(\F))\leq \lfloor l/2 \rfloor$, we are done. If not, $\I^1_{n}(\F)\subseteq \m^{\lfloor l/2 \rfloor+1}$, and hence $\m^{\lfloor l/2 \rfloor}$ annihilates $\I^1_{n}(\F)$. Thus, the submodule $(\m^{\lfloor l/2 \rfloor})^{\oplus \beta_{n}}\subseteq \F_{n}$ is contained in $\ker(d_{n})=\im(d_{n+1})$. This implies that $\m^{\lfloor l/2 \rfloor}\subseteq \I^1_{n+1}(\F)$, and so $\ord(\istab(M,1))\leq \ord(\I^1_{n+1}(\F))\leq \lfloor l/2 \rfloor$.
\end{proof}

\begin{thm}
\label{thm:conductor}
  Let $(R, \m)$ be a 1-dimensional analytically unramified local ring and $\cond_R$ the conductor ideal of $R$. 
  \begin{enumerate}
  \item For any $M\in \C^{+, f}(\mod \- R)$, we have $\cond_R\subseteq \istab(M,1)$. Consequently, $\cond_R\subseteq \istab(\C^{+, f}(\mod \- R),1)=\istab(\mod(R),1)$. \item Assume $R$ is Gorenstein with infinite residue field. The following are equivalent:
  \begin{enumerate}
    \item There is a finitely generated $R$-module $M$ such that $\istab(M,1) = \cond_R$.
        \item $R$ is regular or an abstract hypersurface of multiplicity $2$. 
  \end{enumerate}
  \end{enumerate}
\end{thm}

    \begin{proof}
    Let us prove (1). Let $F$ be the minimal free resolution of $M$, and let $\overline{R}$ and $\widehat{R}$ be as in Lemma~\ref{nagata}. It follows directly from Lemma~\ref{nagata} that $\cond_R\widehat{R}\subseteq \cond_{\widehat{R}}$, which implies $\cond_R \subseteq \cond_{\widehat{R}} \cap R$. Now, suppose $\cond_{\widehat{R}}\subseteq \istab(\widehat{M},1)$. We have
    \begin{align*}
    \cond_R\subseteq \cond_{\widehat{R}}\cap R \subseteq \istab(\widehat{M},1)\cap R
     & = (\bigcap \itot(\widehat{F}_{\geq \inf(F)+i},1))\cap R \\
     & = \bigcap (\itot(\widehat{F}_{\geq \inf(F)+i},1)\cap R)\\
     & = \bigcap (\itot(\F_{\geq \inf(F)+i},1)\widehat{R}\cap R) \\
     & = \bigcap \itot(\F_{\geq \inf(F)+i},1) \\
    &= \istab(M,1),
     \end{align*}
    where the third and fourth equalities hold due to flatness and faithful flatness of the completion map, respectively.
    We may therefore assume $R$ is complete and reduced. Since $R$ is a 1-dimensional reduced ring, it is Cohen-Macaulay. Thus, any first syzygy module over $R$ is maximal Cohen-Macaulay (MCM). By \cite[Proposition 3.1]{Wang1994OnTF}, we have $\cond_R\cdot \Ext^1_R(N,N') = 0$ for any finitely generated MCM $R$-module $N$ and finitely generated $R$-module $N'$. It follows immediately that $\cond_R \subseteq \ca(R)$. 
    Applying \Cref{istab_modR}(1), we conclude that $\cond_R \subseteq \istab(M,1)$; this proves (1).

We now prove (2). We need only consider the case where $R$ is not regular. Since $R$ is analytically unramified, the integral closure $\overline{R}$ of $R$ is a finitely generated $R$-module by \Cref{nagata}(1). The assumptions imply that $\cond_R$ is an $\m$-primary ideal. 

Assume (a), and choose a finitely generated module $M$ such that $\istab(M,1) = \cond_R$.     
Since the residue field of $R$ is infinite, we can choose $x\in \m$ such that it is a minimal reduction of $\m$; see \cite{NR}. We have $x\cond_R=\m\cond_R$, and so $\cond_R\subseteq (x:\m)$. Computing over $A\coloneqq R/xR$, we have $\istab(M/xM,1) = (\cond_R+(x))/(x)\subseteq \soc(A)$. It follows from \Cref{Artinian_socle} that the order of $\soc(A)$ is at most $\lfloor l/2 \rfloor$, where $l$ is the Loewy length of $A$. Since $A$ is Gorenstein, we have $\soc(A) \subseteq \m_A^{l-1}$. We conclude that $l\leq 2$, and so $\soc(A)=\m_A$. Thus, the embedding dimension of $R$, and hence $\widehat{R}$, is two. Since $\widehat{R}$ is also reduced and one-dimensional, it is a hypersurface. Moreover, since $\ell(R/xR)=2$, the multiplicity of $R$ is two \cite[Proposition 11.2.2]{HS}. Thus, $R$ is an abstract hypersurface of multiplicity two.

Assume (b). Suppose $R$ is an abstract hypersurface of multiplicity 2. Let $M=\overline{R}$. As $R$ is Gorenstein, we have $\ell(M/\cond_R M) = 2\ell(R/\cond_R)$ \cite[Theorem 12.2.2]{HS}. As the minimal number of generators of $M$ as an $R$-module is 2, we get that $M/\cond_R M \cong (R/\cond_R)^2$, implying $\I^1_1(M)\subseteq \cond_R$. By \Cref{syzygy_trace}, it follows that $\Tr_R(\syz_1(M))\subseteq \cond_R$. Since $\overline{R}$ is a reflexive $R$-module, we have that $\syz_1(M)$ is an $\overline R$-module (see \cite[2.9]{DMS}). The module $\syz_1(M)$ has rank and is rank 1. It is therefore a regular ideal in $R$ and hence isomorphic to $\overline R$. The minimal free resolution of $\overline{R}$ is therefore one-periodic; by part (1), we conclude that $\istab(M,1)=\cond_R$, as desired.\end{proof}

\subsection{Examples}

We now compute some examples of stable ideals of $1 \times 1$ minors. Throughout this subsection, $k$ denotes a field.

\begin{ex}\label{eg1}
Let $R=k[[x,y]]/(x^2,y^2)$ and $M$ a finitely generated non-free $R$-module. We claim that
\begin{equation}
\label{ex1eq}
\istab(M,1) \in 
\begin{cases}
\{ (x),(y), \m \}, & \on{char}(k) \ne 2; \\
\{ (x),(y), (x+y), \m)\},  &
\on{char}(k) = 2.
\end{cases}
\end{equation}
It is easy to show that each of the ideals on the right side of \eqref{ex1eq} arises as $\istab(M,1)$ for some $M$, and so it follows from \eqref{ex1eq} that $\istab(\mod(R),1)=\soc(R)$. To prove \eqref{ex1eq}, it suffices, since $R$ is Artinian, to determine all possible values of $\itot(M,1)$. Let $I=\I^1_1(M)$.
If $I=\m$, then $\itot(M,1)=\m$. If $I = \m^2$, then $\I^1_2(M) = \m$ and so $\itot(M,1)=\m$. We may therefore assume that $I$ contains exactly one minimal generator of $\m$, say $t$. We have $I = (t)$, and so $\ann(t) \subseteq \I^1_2(M)$, which means $(t)+\ann(t)\subseteq \itot(M,1)$. One easily checks: 
$$
(t)+\ann(t) \in 
\begin{cases}
\{ (x),(y), \m\}, & \on{char}(k) \ne 2; \\
\{ (x),(y), (x+y), \m\},  &
\on{char}(k) = 2.
\end{cases}
$$
\end{ex}

\begin{ex}\label{eg2}
Let $R=k[[t^3,t^4]]\cong k[[x,y]]/(x^4-y^3)$ and $M$ a finitely generated module of infinite projective dimension. We now show that $\istab(M,1)$ is either $(x^2,y)$ or $(x,y)$, and both possibilities occur. We may assume $M$ is MCM and indecomposable. Since $R$ is a hypersurface, $\istab(M,1)$ is equal to $\I^1(A) + \I^1(B)$, where $(A, B)$ is the matrix factorization associated to $M$. Moreover, since $R$ is an ADE singularity, there are only finitely many such matrix factorizations up to isomorphism. One concludes by observing the list of these matrix factorizations in \cite[9.13]{yoshino}.
It follows that $\istab(\mod(R),1)=(x^2,y)$. Observe that $\cond_R = (t^6,t^7,t^8)= (x,y)^2$, which, as predicted by Theorem~\ref{thm:conductor}, is properly contained in $\istab(\mod(R),1) =(x^2,y)$.
\end{ex}

\begin{ex}\label{eg3}
Let $R=k[[t^4,t^5,t^6]]\cong k[[x,y,z]]/(x^3-z^2,xz-y^2)$. We now show that $\istab(\mod(R),1)=(x^2,xy,z)$. Let $M$ be a non-free MCM $R$-module. We have $\cond_R = (t^8, t^9, t^{10}, t^{11}) = (x,y,z)^2$; it therefore follows from Theorem~\ref{thm:conductor}(1) that $\istab(M,1)$ is $\m$-primary. Thus, by Corollary~\ref{Gorenstein_trace}(2), $\istab(M,1)$ is a trace ideal. Given $a \in k$, let $I_a \coloneqq (x-ay, z)$; we conclude from \cite[Example 3.4]{GIK} and Theorem~\ref{thm:conductor}(2) that $\istab(M,1)$ must belong to the list $\{(x,y,z), (x^2,y,z), (x^2,xy,z)\} \cup \{I_a, a\in k\}$. Each of these ideals contains $(x^2,xy,z)$, and so $(x^2,xy,z) \subseteq \istab(\mod(R),1)$. We also have
$
 I_0\cap I_1 = (x^2,xy,z);
$
thus, to show the containment $\istab(\mod(R),1) \subseteq (x^2,xy,z)$, it suffices to show that $\istab(I_a, 1) = I_a$ for $a = 0, 1$. One may see this via an explicit calculation: the minimal $R$-free resolution of $I_0$ is 
$
\left[ R^2 \xla{A} R^2
 \xla{B} R^2 \xla{A} R^2 \xla{B} \cdots \right],
$
where $A  = \begin{pmatrix} -x& -z \\ z & x^2 \end{pmatrix}$ and $B = \begin{pmatrix} -x^2& -z \\ z & x \end{pmatrix}$; and the minimal $R$-free resolution of $I_1$ is $
\left[ R^2 \xla{A'} R^2
 \xla{B'} R^2 \xla{A'} R^2 \xla{B'} \cdots \right],
$
where we have $A'  = \begin{pmatrix} y-x &x^2 -z \\ z & x^2 + xy \end{pmatrix}$ and $B' = \begin{pmatrix} -x^2 - xy& x^2-z \\ z & x-y \end{pmatrix}$. Finally, we observe once again that $\cond_R$ is properly contained in $\istab(\mod(R),1)$, in accordance with Theorem~\ref{thm:conductor}.
\end{ex}

\section{Periodicity of ideals of minors over complete intersections}\label{CI_section}
We now study the eventual behavior of ideals of minors in free resolutions over complete intersections. In \S\ref{sec:shamashbackground}, we recall the notion of a Shamash resolution, and we explain how Shamash resolutions can be interpreted in terms of matrix factorizations. We prove in \S\ref{sec:shamashres} (resp. \S\ref{periodicity_ideal_of_minors}) that the  ideals of minors in the Shamash (resp. minimal free) resolution of a finitely generated module $M$ over complete intersections are eventually 2-periodic. In the case of a Shamash resolution, we obtain a bound on the homological degree at which this periodicity starts that is independent of $M$, mirroring the behavior of minimal free resolutions over hypersurfaces (Corollary~\ref{cor:easy}). For minimal free resolutions, we show that this is impossible (Example~\ref{uniform_bound_example}). 

Let us establish some notation for this section. Let $S$ be a commutative ring, $f_1,\dots,f_c\in S$, and $R=S/(f_1,\dots,f_c)$. Let $\tS$ denote the graded ring $S[t_1, \dots, t_c]$, where $\deg(t_i) = -2$, and $\tD$ the graded $S$-module $\Hom_S(\tS, S)$. Observe that $\tD$ is the divided power algebra over $S$ on the degree $2$ dual variables $y_1, \dots, y_c$ corresponding to $t_1, \dots, t_c$; moreover, $\tD$ is an $\tS$-module via contraction.

\subsection{The Shamash construction and matrix factorizations}
\label{sec:shamashbackground}
We begin by recalling the Shamash construction on a complex of $S$-modules associated to $f_1, \dots, f_c \in S$, introduced by Shamash in the case where $c = 1$ \cite{SHAMASH19711} and generalized by Eisenbud to the $c \ge 1$ case \cite{eisenbudenriched}. Our reference is Eisenbud-Peeva's book \cite{EP}. 

Given $\a \in \Z^c_{\ge 0}$, write $t^{\a} \coloneqq t_1^{a_1} \cdots t_c^{a_c}$, and let $|\a| = \sum_{i = 1}^c a_i$.

\begin{dfn}[\cite{EP} Definition 3.4.1]\label{Shamash_def}
Let $G$ be a complex of free $S$-modules. A \emph{system of higher homotopies for $f_1,\dots,f_c$ on $G$} is a collection $\sigma$ of morphisms\footnote{Our definition looks slightly different from \cite[Definition 3.4.1]{EP} at first glance, as the latter involves morphisms $G \to G[-2|a|+1]$. This is because we are using a different convention for homological shifts: in \cite{EP}, $C[1]_i \coloneqq C_{i - 1}$, while for us, $C[1]_i \coloneqq C_{i + 1}$.  }
\[\sigma_{\a}:G\rightarrow G[2|\a|-1]\]
of graded modules (not complexes) for each $\a = (a_1, \dots, a_c) \in \Z^c_{\ge 0}$ with the following properties:
\begin{enumerate}
\item[(a)] $\sigma_0$ is the differential on $G$. 
\item[(b)] Letting $e_i$ denote the $i^{\th}$ standard basis vector for each $i$, the map $\sigma_0\sigma_{e_i}+\sigma_{e_i}\sigma_0$ is given by multiplication by $f_i$ on $G$ for each $1\leq i\leq c$.
\item[(c)] If $|a|  \geq 2$, then
$\sum_{\u+\mathbf{v}=\a}\sigma_{\u}\sigma_{\mathbf{v}}=0$.
\end{enumerate}
\end{dfn}

\begin{constr}\label{shamash_construction}
Let $G$ be a complex of free $S$-modules equipped with a system of higher homotopies $\sigma$. The \emph{Shamash construction} $\sh(G,\sigma)$ associated to the pair $(G, \sigma)$ has underlying module $\tD\otimes_S G \otimes_S R$ and differential $\mathlarger{\sum} t^{\a}\otimes \sigma_{\a} \otimes \id$.
\end{constr}
\begin{prop}[\cite{EP} Propositions 3.4.2 and 4.1.4]\label{higher_homotopies}
 Let $G$ be an $S$-free resolution of a finitely generated $R$-module $M$.
\begin{enumerate}
    \item There exists a system of higher homotopies on $G$ for $f_1,\dots,f_c$.
    \item Assume that $f_1, \dots, f_c \in S$ form a regular sequence. If $\sigma$ is a system of higher homotopies on $G$ for $f_1,\dots,f_c$, then $\sh(G,\sigma)$ is an $R$-free resolution of $M$.
\end{enumerate}
\end{prop}

When the complex $\sh(G, \sigma)$ in Construction \ref{shamash_construction} is an $R$-free resolution of a module $M$, we call it a \emph{Shamash resolution} of $M$. 

\begin{rem}
If, in the context of Construction \ref{shamash_construction}, the maps $\sigma_{\a}$ in the system of higher homotopies are all minimal, then $\sh(G, \sigma)$ is also minimal. In the setting of Proposition~\ref{higher_homotopies}(b), it is not always possible to choose a minimal system of higher homotopies; that is, minimal free resolutions of complete intersections are not always Shamash resolutions. However, the minimal free resolution of the residue field of a complete intersection is always a Shamash resolution \cite{gulliksen, Tate_57}.
\end{rem}

\begin{ex}
\label{ex:shamash1}
Let $k$ be a field, $S = k[x,y,z,w]$, $f_1 = xz$, and $f_2 = yw$. The minimal $S$-free resolution of $S / (x,y)$ is the Koszul complex 
$$
G = (S \xla{\begin{pmatrix} x & y \end{pmatrix}} S^2 \xla{\begin{pmatrix} -y \\ x \end{pmatrix}} S).
$$
A system of higher homotopies on $G$ is given as follows. Take $\sigma_{e_1}$ (resp. $\sigma_{e_2}$) to be the nullhomotopy of multiplication by $xz$ (resp. $yw$) given by the maps $\begin{pmatrix} z \\ 0 \end{pmatrix}$ and $\begin{pmatrix} 0 & z \end{pmatrix}$ (resp. $\begin{pmatrix} 0 \\ w \end{pmatrix}$ and $\begin{pmatrix} -w & 0 \end{pmatrix}$). For degree reasons, all the other $\sigma_{\a}$ must be 0. One easily checks that these maps form a system of higher homotopies. The complex $\sh(G, \sigma)$ is given by $R^{i+1}$ in homological degree $i$:
\begin{equation}
\label{ex:resolution}
R \xla{\begin{pmatrix} x & y \end{pmatrix}}R^2 
\xla{\begin{pmatrix} -y & z & 0 \\ x & 0 & w \end{pmatrix}} 
R^3
\xla{\begin{pmatrix} 0&z&-w&0 \\ x&y&0&0 \\ 0&0&x&y \end{pmatrix}}
R^4 \from \cdots
\end{equation}
Since the maps $\sigma_{\a}$ in our system of higher homotopies are minimal, this complex is the minimal $R$-free resolution of $S / (x, y)$.
\end{ex}

We now give an alternative take on the Shamash construction in the language of matrix factorizations, as this perspective will be useful for our study of ideals of minors in Shamash resolutions in \S\ref{sec:shamashres}. Our approach relies heavily on work of Martin~\cite{martin}.

\begin{dfn}
A \emph{graded matrix factorization} of $\tf = f_1t_1 + \cdots + f_ct_c \in \tS$ is a pair $(F, d)$, where $F$ is a finitely generated, graded, free $\tS$-module, and
$d$ is a degree $-1$ endomorphism of $F$ such that $d^2 = \tf \cdot \id_F$.
A \emph{morphism} of graded matrix factorizations is a degree 0 endomorphism of $F$ that commutes with $d$, and a \emph{homotopy} of such morphisms $f$ and $g$ is a degree $1$ endomorphism $h$ of $F$ such that $f- g = dh + hd$. Let $\mf(\tS, \tf)$ denote the category of graded matrix factorizations of $\tf$ and $\hmf(\tS, \tf)$ the quotient of $\mf(\tS, \tf)$ given by modding out null-homotopic morphisms.
\end{dfn}

\begin{ex}
\label{ex:mf}
Let $G$ be a bounded complex of finitely generated free $S$-modules equipped with a system of higher homotopies $\sigma$ for $f_1, \dots, f_c$. Letting $F = \bigoplus_{i \in \Z} \tS(-i) \otimes_S G_i$ and $d = \sum t^{\a} \otimes \sigma_{\a}$ gives an object $(F, d)$ in $\mf(\tS, \tf)$. 

For instance, applying this to the complex in Example~\ref{ex:shamash1}, we get $F = \tS(-1)^2 \oplus \tS \oplus \tS(-2)$ (the terms are out of numerical order so that the components generated in odd and even degrees are separated) and
$d = \begin{pmatrix} 0 & A \\ B & 0 \end{pmatrix}$, where $A = \begin{pmatrix} zt_1 & -y \\ wt_2 & x \end{pmatrix}$ and $B = \begin{pmatrix} x & y \\ -wt_2 & zt_1 \end{pmatrix}$. 
\end{ex}

By a result of Burke-Stevenson \cite[Theorem 7.5]{BS}, when $S$ is regular of finite Krull dimension and $f_1, \dots, f_c$ is a regular sequence, there is an equivalence of categories
$$
\hmf(\tS, \tf) \xra{\simeq} \Db(R)
$$
(see also \cite[Theorem 1]{BW} for a related result). Burke-Stevenson's proof makes crucial use of a theorem of Orlov concerning singularity categories of graded Gorenstein algebras \cite[Theorem 2.5]{orlov}. Martin gives an alternative formulation of this equivalence in his thesis \cite{martin} via a version of the BGG correspondence; we note that Martin's result does not require $S$ to have finite Krull dimension. As this equivalence is not expressed in \cite{martin} in quite the way we need it, we reformulate Martin's result in the following way. 

We consider the functor $\mf(\tS, \tf) \to \on{D}(R)$ that sends a matrix factorization $(F, d)$ to the complex $(F \otimes_{\tS} R, d)$ of $R$-modules; the component in homological degree $i$ of $F\otimes_{\tS} R$ is  $F_{i} \otimes_S R$. This functor respects homotopy and therefore induces a functor
\begin{equation}
\label{eq:functor}
\hmf(\tS, \tf) \to \on{D}(R).
\end{equation}
Given $(F,d) \in \mf(\tS, \tf)$, let $F^\vee$ denote the $S$-linear dual of $F$. We also define a functor
$$
\Phi : \hmf(\tS, \tf) \to \on{D}(R)^{\op}
$$
that sends an object $(F, d)$ to the complex $F^\vee \otimes_S R$ with differential $d^T \otimes \id$. That is, $\Phi$ is given by dualizing over $S$ and tensoring with $R$; equivalently, $\Phi$ is given by applying \eqref{eq:functor} and dualizing over $R$.

The following theorem is a nearly immediate consequence of work of Martin \cite{martin}:

\begin{thm}[\cite{martin}]
\label{equivalence_mf_D^b(R)}
Let $S$ be a regular ring, and assume $f_1, \dots, f_c$ form a regular sequence in $S$. 
The functor \eqref{eq:functor} induces an equivalence $\hmf(\tS, \tf) \xra{\simeq} \Db(R)$. If $R$ has finite Krull dimension, then $\Phi$ determines an equivalence $\hmf(\tS, \tf) \xra{\simeq} \Db(R)^{\op}$.
\end{thm}

\begin{proof}
Let $K$ be the Koszul complex on $f_1, \dots, f_c$, with differential denoted $d_K$. We view $K$ as an exterior algebra $\bigwedge_S(e_1, \dots, e_c)$. Let $X \in \mf(\tS, \tf)$ be the object with underlying module $\bigoplus_{i = 0}^c \tS(-i) \otimes_S K_i$ and differential $d_X = (1 \otimes d_K) + \lambda$, where $\lambda$ is given by left multiplication by the element $\sum_{i = 1}^c t_i \otimes e_i$. By \cite[Theorem 5.1]{martin}, there is an equivalence $\hmf(\tS, \tf) \xra{\simeq} \Db(K)$ that sends a matrix factorization $(F, d_F)$ to the complex $\Hom_{\tS}(X, F)$ with differential $  \alpha d_X - (-1)^{\deg(\alpha)}d_F \alpha$.
The differential squares to 0 since $d_X$ and $d_F$ both square to multiplication by $\tf$. Since $f_1, \dots f_c$ is a regular sequence, extension of scalars along the quasi-isomorphism $K \xra{\simeq} R$ induces an equivalence $\Db(K) \xra{\simeq} \Db(R)$; composing, we obtain an equivalence
$
\hmf(\tS, \tf) \xra{\simeq} \Db(R).
$
Tracing through the formulas, one sees that this equivalence sends $(F, d_F)$ to the complex $(F \otimes_{\tS} R,d)$, and so \eqref{eq:functor} determines an equivalence as stated. Since $R$ is Gorenstein and has finite Krull dimension, the last statement is immediate.
\end{proof}

The connection between matrix factorizations and Shamash resolutions is provided by the following result. Before stating it, we note that, given a complex $G$ of free $S$-modules and a system of higher homotopies $\sigma_{\a} : G \to G[2|\a| - 1]$ for $f_1, \dots, f_c$, the maps $\sigma_{\a}^\vee : G^\vee \to G^\vee[2|\a| -1]$ determine a system of higher homotopies for $f_1, \dots, f_c$ on $G^\vee$.

\begin{prop}
\label{rem:shamash}
Let $G$, $\sigma$, and $(F, d)$ be as in Example~\ref{ex:mf}. Applying $\Phi$ to $(F, d)$ gives the Shamash construction $\sh(G^\vee, \sigma^\vee)$.
\end{prop}

\begin{proof}
This is a straightforward calculation.
\end{proof}

Motivated by Proposition~\ref{rem:shamash}, we make the following
\begin{defn}
We call a complex of $R$-modules a \emph{Shamash complex} if it is isomorphic (as a complex, not just in $\on{D}(R)$) to $\Phi(F, d)$ for some object $(F, d) \in \mf(\tS, \tf)$.
\end{defn}

\subsection{Periodicity of ideals of minors in Shamash complexes}
\label{sec:shamashres}

We now prove that ideals of minors in Shamash complexes are eventually 2-periodic (Theorem~\ref{shamashperiodicity}). Moreover, we give an explicit upper bound on the homological index where the periodicity begins. For ideals of $1 \times 1$ minors of Shamash resolutions of modules over complete intersections, the periodicity begins after $\dim(S)$+1 steps, mirroring the starting point for the periodicity of minimal free resolutions of modules over hypersurfaces (Corollary~\ref{cor:eisenbud}).

We begin with a key technical lemma. We will need the following

\begin{nota}
\label{thetanotation}
Let $F$ be a graded free $\tS$-module. Choose a basis of $F$, so that we may write $F = \bigoplus \tS(-d_i)$. Given $i \in \Z$, the graded component $F_i$ has a $S$-basis $B_i$ indexed by monomials in $\tS$; notice that the $B_i$ are determined canonically by our choice of basis for $F$. We set
$$
\theta_F(i) \coloneqq \on{sup} \{j : \text{$t_1^j$ divides an element of $B_i$} \}.
$$
If $F$ is bounded above, i.e. $F_\ell = 0$ for all $\ell \gg 0$, then $\theta_F(i) < \infty$ for all $i$.
\end{nota}

\begin{lem}
\label{lem:technical}
Assume $c > 1$. Let $G$ and $G'$ be graded free $\tS$-modules that are bounded above and equipped with bases; moreover, assume $G$ is generated entirely in even (resp. odd) degree, and $G'$ is generated entirely in odd (resp. even) degree. Let $A: G \to G'$ be an $\tS$-linear map that is homogeneous of degree $-1$. Denote by $A_i : G_i \to G_{i -1}'$ the $S$-linear map induced by $A$. Suppose there is some even (resp. odd) integer $n \ll 0$ such that $\I^1(A_i) = \I^1(A_{i-2})$ for all $i \le n$. Fix $r \ge 1$; for any even (resp. odd) integer $\ell$ such that $\ell \le -2(r-1)(\max\{\theta_G(n), \theta_{G'}(n-1)\} + 1) + n$ (see Notation~\ref{thetanotation}), we have $\I^r(A_\ell) = \I^1(A_n)^r$.
\end{lem}

\begin{proof}
It is immediate that $\I^r(A_\ell) \subseteq \I^1(A_\ell)^r = \I^1(A_n)^r$; we now prove the opposite containment. We view each $A_i : G_i \to G'_{i-1}$ as a matrix with entries in $S$ via the monomial bases of $G_i$ and $G'_{i-1}$ induced by our chosen bases of $G$ and $G'$. Let $x_1, \dots, x_r$ be entries of $A_n$; it suffices to exhibit an $r \times r$ submatrix of $A_\ell$ with determinant $x_1 \cdots x_r$. Each of our entries $x_k$ of $A_n$ corresponds to monomial basis elements $m_k$ and $n_k$ of $G_n$ and $G'_{n-1}$, respectively. Let $s \coloneqq (n - \ell)/2
$, $p \coloneqq \max\{\theta_G(n), \theta_{G'}(n-1)\} + 1$, and $g_i \coloneqq t_1^{s - (i-1)p}t_2^{(i-1)p} \in \tS$ for $1 \le i \le r$. Notice that each $s - (i - 1)p$ is indeed nonnegative, by our assumption on $\ell$. Observe also that the degree of each $g_i$ is $-(n-\ell)$.

By our choice of $p$, we have $m_ig_i = m_jg_j$ if and only if $i = j$. Indeed, if $m_ig_i = m_jg_j$ for $i > j$, then $m_it_1^{(i - j)p} = m_jt_2^{(i -j)p}$; this means $t_1^p$ divides $m_j$, which is impossible. Similarly, $n_ig_i = n_jg_j$ implies $i = j$. We therefore have the following $r \times r$ submatrix $P$ of $A_\ell$:
$$
    \bordermatrix{~ & m_1g_1 & m_2g_2& \cdots & m_{r-1}g_{r-1} & m_rg_{r}\cr
                  n_1g_1 & p_{1,1} & p_{1,2} & \cdots & p_{1,r-1}& p_{1, r}\cr
                  n_2g_2 & p_{2,1} & p_{2,2} & \cdots & p_{2,r-1}& p_{2, r}\cr
                  \vdots & \vdots & \vdots & \cdots & \vdots & \vdots \cr
                  n_{r-1}g_{r-1} & p_{r-1,1} & p_{r-1,2} & \cdots & p_{r-1,r-1}& p_{r-1, r}\cr
                  n_rg_r & p_{r,1} & p_{r,2} & \cdots & p_{r,r-1}& p_{r, r}\cr
                  }.
$$  
Since $A$ is $\tS$-linear, we have $p_{i, i} = x_i$ for all $i$. It thus suffices to show that $P$ is lower triangular (in fact, we will see that it is diagonal). We have $A(m_jg_j) = g_jA(m_j)$, which is a linear combination of the form $\sum_{k} a_kg_jb_k$, where each $a_k$ is in $S$, and each $b_k$ is a basis element of $G'_{n-1}$. It is impossible that $g_jb_k = g_in_i$ for some $i < j$, since $t_1^{s-(i-1)p}$ divides the right side but not the left side. It follows that $p_{i, j} = 0$ for all $i < j$ (in fact, a similar argument shows $p_{i, j} = 0$ for all $i > j$ as well, i.e. $P$ is diagonal).
\end{proof}

\begin{prop}
\label{mfperiodicity}
Assume $c > 1$. Let $(F, d) \in \mf(\tS, \tf)$, and fix $r \ge 1$. Write $d_j : F_j \to F_{j-1}$ for the degree $j$ component of $d$. Let $N_0$ (resp. $N_1$) be the largest even (resp. odd) degree in which $F$ is generated and $n_0$ (resp. $n_1$) the smallest even (resp. odd) such degree. Set
$$
M_i \coloneqq -(r-1)(N_i - n_i + 4) + n_i, \quad i = 0, 1. 
$$
\begin{enumerate}
\item If $\ell \in \Z$ is even, and $\ell \le M_0$, then $\I^r(d_\ell) = \I^1(d_{n_0})^r$.
\item If $\ell \in \Z$ is odd, and $\ell \le M_1$, then $\I^r(d_\ell) = \I^1(d_{n_1})^r$.
\end{enumerate}
\end{prop}

\begin{proof}
Let us prove (1); the proof of (2) is the same. Choose a basis of $F$, so that each graded component $F_i$ may be equipped with a monomial basis over $S$. Let $F_{\even}$ (resp. $F_{\odd}$) be the submodule of $F$ given by even (resp. odd) degree elements. We will apply Lemma~\ref{lem:technical} to the $\tS$-linear map $d : F_{\even} \to F_{\odd}$. 

By $\tS$-linearity, it is clear that $\I^1(d_i) = \I^1(d_{i-2})$ for all $i \le n_0$. Notice that $\theta_{F_{\even}}(n_0) = \frac{N_0 - n_0}{2}$. Let us now show $\theta_{F_{\odd}}(n_0-1) \le \theta_{F_{\even}}(n_0) + 1$. If $\theta_{F_{\odd}}(n_0-1) = 0$, we are done, so suppose we have a monomial basis element $m \in F_{n_0 - 1}$ of the form $m = t_1^qm'$ for some $q \ge 1$ and monomial $m'$. Write $d(m') = a_1m_1 + \cdots + a_tm_t$, where $a_i \in S$ and each $m_i$ is an element of the monomial basis of $F_{n_0 + 2q - 2}$. We have $t_1^{q-1}m_i \in F_{n_0}$ for all $i$, and so $q - 1 \le \theta_{F_{\even}}(n_0)$, which implies the inequality we seek.

Putting everything together, we have:
\begin{align*}
\ell & \le -(r-1)(N_0 - n_0 + 4) + n_0 \\
& = -2(r-1) (\theta_{F_{\even}}(n_0) + 2) + n_0 \\
& \le -2(r-1)(\max\{\theta_{F_{\even}}(n_0), \theta_{F_{\odd}}(n_0 -1)\} + 1) + n_0.
\end{align*}
Applying Lemma~\ref{lem:technical} finishes the proof.
\end{proof}

\begin{thm}\label{shamashperiodicity}
Assume $c > 1$, and let $(F, d) \in \mf(\tS, \tf)$. Let $n_i$ and $M_i$ for $i = 0, 1$ be as in Proposition~\ref{mfperiodicity}. For all $r \ge 1$ and $\ell \in \Z$, we have:
\begin{enumerate}
\item If $\ell$ is even, and $\ell  \ge -M_1 + 1$, then $\I^r_\ell(\Phi(F, d)) = \I^1_{-n_1 + 1}(\Phi(F, d))^r$. 
\item If $\ell$ is odd, and $\ell  \ge -M_0 + 1$, then $\I^r_\ell(\Phi(F, d)) = \I^1_{-n_0 + 1}(\Phi(F, d))^r$.
\end{enumerate}
In particular, the ideals of minors in Shamash complexes are eventually 2-periodic.
\end{thm}

\begin{proof}
Let us prove (1); the proof of (2) is the same. We have $-\ell + 1 \le M_1$, so Proposition~\ref{mfperiodicity}(2) implies that $\I^r(d_{-\ell + 1}) = \I^1(d_{n_1})^r$. Dualizing, we get $\I^r(d^T_{\ell}) = \I^1(d^T_{-n_1 + 1})^r$. This equality still holds upon tensoring the arguments with $R$;  the result follows.
\end{proof}

\begin{ex}
Let $G$ and $\sigma$ be as in Example~\ref{ex:shamash1}, and let $(F, d)$ be the matrix factorization obtained from $(G^\vee, \sigma^\vee)$ as in Example~\ref{ex:mf}. By Proposition~\ref{rem:shamash}, the Shamash resolution $C$ constructed in Example~\ref{ex:shamash1} is given by $\Phi(F, d)$. Let us now analyze the minors of $C$ via Theorem~\ref{shamashperiodicity}.

Fix $r \ge 1$. We have $n_0 = -2$, $N_0 = 0$, and $n_1 = -1 = N_1$; thus, $M_0 = -6(r-1)-2$, and $M_1 = -4(r-1)-1$. Theorem~\ref{shamashperiodicity} therefore implies that the ideals of minors in odd (resp. even) homological degrees in $C$ become 2-periodic starting---at most---at position $6(r - 1) + 3$ (resp. $4(r - 1) + 2$). More specifically: the ideals of $r \times r$ minors in odd (resp. even) homological degrees stabilize to $(x, y, z,w)^r$ at position $6(r-1) + 3$ (resp. $4(r-1) + 2$). When $r = 1$, these bounds are sharp. The bounds are not sharp for higher ranks; a reason for this is that the bounds for periodicity we obtain in Proposition~\ref{mfperiodicity} are for ideals of minors in the ring $S$; modding out by $f_1, \dots, f_c$ often causes additional vanishing of minors.
\end{ex}

Assume now that $S$ is a regular local ring, and $f_1, \dots, f_c$ is a regular sequence in $S$. Let $M$ be a finitely generated $R$-module and $G$ its minimal $S$-free resolution. By Proposition~\ref{higher_homotopies}, we may equip $G$ with a system of higher homotopies $\sigma$ for $f_1, \dots, f_c$. As discussed above Proposition~\ref{rem:shamash}, dualizing yields a system of higher homotopies $\sigma^\vee$ on $G^\vee$. By Proposition~\ref{rem:shamash}, applying the construction in Example~\ref{ex:mf} to $(G^\vee, \sigma^\vee)$ yields a matrix factorization $(F, d)$ such that $\Phi(F, d)$ is the Shamash resolution $\sh(G, \sigma)$ of $M$. The numbers $N_0$ and $n_0$ (resp. $N_1$ and $n_1$) from Proposition~\ref{mfperiodicity} in this case are precisely the smallest and largest even (resp. odd) nonvanishing homological degrees of $G$. In particular, we have $0 \le -N_i \le -n_i \le \pd_S(M)$ for $i = 0, 1$, and so Theorem~\ref{shamashperiodicity} implies:

\begin{thm}
\label{thm:cishamash}
Let $S$ be a regular local ring, $f_1, \dots, f_c$ a regular sequence in $S$ with $c > 1$, $M$ a finitely generated $R = S/(f_1, \dots, f_c)$-module, $k$ the residue field of $R$, and $\sh(G, \sigma)$ a Shamash resolution of $M$. Fix $r \ge 1$. We have:
\begin{enumerate}
\item $
\I^r_\ell(\sh(G, \sigma)) = \I^r_{\ell+2}(\sh(G, \sigma))
$
for all $\ell \ge r \cdot \pd_S(M) + 4(r-1) + 1.$
\item $
\I^r_\ell(k) = \I^r_{\ell+2}(k)
$
for all $\ell \ge r \cdot c + 4(r-1) + 1$.
\end{enumerate}
\end{thm}

\begin{proof}
Part (1) is a consequence of Theorem~\ref{shamashperiodicity} and the inequalities $0 \le -N_i \le -n_i \le \pd_S(M)$. Part (2) follows since the minimal $R$-free resolution of $k$ is given by the Shamash construction \cite{gulliksen, Tate_57}. 
\end{proof}

\begin{remark}
\label{rem:necessary}
The bounds obtained in Theorem~\ref{thm:cishamash} depend (linearly) on the rank of the minors; we note that such dependence is necessary. For instance, in the setting of the Theorem, let $F$ be the minimal free resolution of the residue field $k$. Notice that $\I^r_\ell(k) = 0$ when $\on{rank} F_\ell < r$. However, by Theorem~\ref{residue_field_stable_ideals}(2), for any $r > 0$, there exists $\ell \gg 0$ such that $\I^r_\ell(k) \ne 0$. Thus, the point at which the ideals $\I^r_\ell(k)$ stabilize depends on $r$.
\par Even if the higher rank minors do not vanish, they may not stabilize when the lower rank minors do. For example, in \Cref{ex:shamash1}, we have $\I^1_i(\sh(G,\sigma))=\I^1_2(\sh(G,\sigma))$ for all even integers $i\geq 2$, but $0\neq \I^2_2(\sh(G,\sigma))\neq \I^2_4(\sh(G,\sigma))$.
\end{remark}

As a consequence, we obtain a uniform upper bound (i.e. independent of the choice of module) for where the periodicity of the ideals of minors of a Shamash resolution begins:
\begin{cor}
\label{cor:easy}
In the setting of Theorem~\ref{thm:cishamash}, we have $
\I^r_\ell(\sh(G, \sigma)) = \I^r_{\ell+2}(\sh(G, \sigma))
$
for $\ell \ge r \cdot \dim(S) + 4(r-1) + 1.$
\end{cor}

A famous result of Eisenbud \cite[Theorem 6.1(i)]{Eisenbud_1980} states that minimal free resolutions of finitely generated modules over local hypersurface rings become 2-periodic after at most $\dim(Q) + 1$ steps, where $Q$ is the ambient regular ring. Since minimal free resolutions over hypersurfaces are given by the Shamash construction, the following Corollary extends this result, on the level of rank 1 minors, to complete intersections:

\begin{cor}
\label{cor:eisenbud}
Let $S$ be a regular local ring, $f_1, \dots, f_c$ a regular sequence in $S$, $M$ a finitely generated $R = S/(f_1, \dots, f_c)$-module, and $\sh(G, \sigma)$ a Shamash resolution of $M$. We have
$
\I^1_\ell(\sh(G, \sigma)) = \I^1_{\ell+2}(\sh(G, \sigma))
$
for all $\ell \ge \dim(S) + 1.$ 
\end{cor}

\begin{proof}
Immediate from Corollary~\ref{cor:easy} and \cite[Theorem 6.1(i)]{Eisenbud_1980}. 
\end{proof}

We conclude this subsection with another periodicity result for Shamash complexes that will be useful in \S\ref{periodicity_ideal_of_minors}.

\begin{prop}
\label{prop:shamashc1}
Let $S$ be a commutative ring, $f \in S$, $R = S/(f)$, $G$ a complex of finitely generated free $R$-modules, $\sigma$ a system of higher homotopies on $G$ for $f$, and $F = \sh(G, \sigma)$. Assume $R$ is Noetherian. For $r > 0$, we have $\I^r_k(F) = \I^r_{k+2}(F)$ for $k \gg 0$.
\end{prop}

\begin{proof}
Let $\widetilde{S} = S[t]$, with $t$ a degree $-2$ variable. We have $F = \widetilde{D} \otimes_S G \otimes_S R$, where $\widetilde{D}$ is the divided power algebra $\Hom_S(\widetilde{S}, S)$, and the differential $d$ on $F$ is given by $\sum_{i \ge 0} t^i \otimes \sigma_i \otimes \id$. Let $y \in \widetilde{D}$ denote the dual of $t$. For all $k \in \Z$, $F_{k+2}$ decomposes as $yF_{k} \oplus G_{k+2}$. The differential $d_{k+2}$ is a square matrix with respect to these decompositions, with top-left entry $d_{k} \: yF_k \to yF_{k-1}$.
Thus, $\I^r_k(F) \subseteq \I^r_{k+2}(F)$ for all $k \in \Z$. The Noetherianity of $R$ implies the result.
\end{proof}

\subsection{Periodicity of ideals of minors of minimal free resolutions over complete intersections}\label{periodicity_ideal_of_minors}

Corollary~\ref{cor:easy} gives a tidy description of the periodic behavior of the ideals of $r \times r$ minors of a Shamash resolution for a module $M$ over a complete intersection: there is a upper bound---independent of $M$---for when the periodicity begins, and this bound depends linearly on $r$. Moreover, when $r = 1$, this phenomenon is an extension of the familiar periodic behavior of minimal free resolutions over hypersurfaces (Corollary~\ref{cor:eisenbud}). 

One might expect that the situation is the same for \emph{minimal} free resolutions over complete intersections, but the story turns out to be more complicated. On one hand, we prove in Theorem~\ref{CI_periodicity} that the ideals of minors of minimal free resolutions over complete intersections are indeed eventually 2-periodic. On the other hand, there is no bound on where the periodicity starts that is independent of the module (Example~\ref{uniform_bound_example}); in fact, we do not obtain any upper bound at all on where periodicity begins. 

Our main tool for establishing the eventual periodicity of ideals of minors of minimal free resolutions over complete intersections (Theorem~\ref{CI_periodicity}) is Eisenbud-Peeva's theory of higher matrix factorizations, as introduced in their book \cite{EP}. The definition of a higher matrix factorization, and the sense in which higher matrix factorizations govern the asymptotic behavior of minimal free resolutions over complete intersections, is quite intricate; but the details will not be necessary for this paper. The key point is that, as discussed in \cite[Construction 5.1.1]{EP}, the minimal free resolution of a higher matrix factorization module over a complete intersection of codimension $c$ is given by a Shamash construction associated to a single element: this allows us to apply Proposition~\ref{prop:shamashc1}. 

With this in mind, let us now prove our periodicity result for ideals of minors of minimal free resolutions over complete intersections. In fact, we work in slightly greater generality:

\begin{thm}\label{CI_periodicity}
Let $S$ be a local ring, $f_1, \dots, f_c \in S$ a regular sequence, $R$ the quotient $S/(f_1,\dots,f_c)$, and $M$ a finitely generated $R$-module such that $\pd_S(M) < \infty$. Given $r>0$, we have $\I^r_k(M) = \I^r_{k + 2}(M)$ for $k\gg 0$.
\end{thm}

\begin{proof}
By \cite[Corollary 6.4.3]{EP}, a sufficiently high syzygy of $M$ is a minimal higher matrix factorization (HMF) module with respect to some choice of generators $f_1',\dots,f_c'$ of the ideal $(f_1,\dots,f_c)$. Thus, truncating and replacing each $f_i$ with $f_i'$, we may assume $M$ is a minimal HMF module for $f_1, \dots, f_c$. By \cite[Construction 5.1.1]{EP}, the minimal free resolution of $M$ is given by applying a Shamash construction on $f_c$ to the minimal free resolution of $M$ over $S/(f_1, \dots, f_{c-1})$. Applying Proposition~\ref{prop:shamashc1} finishes the proof.
\end{proof}

\begin{remark}
\label{ciremarks}
By the faithful flatness of completion, \Cref{CI_periodicity} holds for modules over a ring $R$ whose completion at its maximal ideal is isomorphic to a regular local ring modulo a regular sequence. This hypothesis is indeed weaker than assuming $R$ itself is a quotient of a regular local ring by a regular sequence: see \cite{HJ}.

Additionally, the statement of \Cref{CI_periodicity} holds when $S$ is a regular local ring and $M$ is a bounded below complex of finitely generated free $R$-modules with homology concentrated only in finitely many degrees. Indeed, a smart truncation of the minimal free resolution of such a complex is a minimal free resolution of a module, and so one immediately reduces to the setting of \Cref{CI_periodicity}.
\end{remark}

\begin{rem}\label{ideal_of_minors_containment}
    The proofs of Proposition~\ref{prop:shamashc1} and Theorem~\ref{CI_periodicity} show a bit more. Let $M$ be a finitely generated module over a local complete intersection. As discussed in the proof of Theorem~\ref{CI_periodicity}, it follows from results of \cite{EP} that we may choose $i \gg 0$ such that $\syz_i(M)$ is a higher matrix factorization module associated to some regular sequence and is therefore resolved by a Shamash construction on a single element. It therefore follows from the proof of Proposition~\ref{prop:shamashc1} that $\I^r_j(M) \subseteq \I^r_{j+2}(M)$ for all $r \ge 1$ and $j \ge i$.
\end{rem}

The following example shows that, in the setup of Theorem~\ref{CI_periodicity}, there is no bound on where the periodicity begins that is independent of $M$.

\begin{ex}\label{uniform_bound_example}
Let $(S,\m,k)$ be a regular local ring, $R = S/I$ an Artinian complete intersection of codimension at least two, and $F$ the minimal $R$-free resolution of $k$. Let $r$ be an integer greater than $\beta_3(k)$, the third Betti number of $k$. We first show that there is some $\ell > 0$ such that $\I^r_\ell(k) \ne \I^r_{\ell+2}(k)$.
Notice that $\I^r_2(k) = \I^r_1(k) = 0$. On the other hand, by Theorem~\ref{residue_field_stable_ideals}, $\istab(k, r) = \m^r$, and so there is some $i$ such that $\I^r_i(k) \ne 0$. The existence of $\ell > 0$ with the desired property follows.

Let $F^\vee$ denote the complex $\Hom_R(F, R)$. Since $R$ is self-injective, $F^\vee$ is a projective co-resolution of $k^\vee \cong k$. The mapping cone $C$ of the composition $F \onto k \cong k^\vee \into F^\vee$ is therefore an exact, minimal complex of finitely generated free $R$-modules. (The complex $C$ is called a \emph{complete resolution} of $k$ \cite{buchweitz}.)
Choosing $\ell$ as in the preceding paragraph, we see that taking a smart truncation of $C$ in arbitrarily negative degree gives a minimal free resolution of an $R$-module whose ideals of minors are not equal in arbitrarily large even/odd degrees.
\end{ex}

Let $M$ be a finitely generated module over a local ring. The \emph{$j^{\th}$ Fitting ideal of $M$ in homological degree $i$} is the ideal $\mathbf{F}^j_i(M) \coloneqq \I^{\beta_{i-1}-j+1}_i(M)$, where $\beta_\ell$ denotes the $\ell^{\th}$ Betti number of $M$. When $i = 1$, this recovers the definition of the $j^{\th}$ Fitting ideal of $M$ in \cite{BE_77}. Fitting ideals in large homological degrees over complete intersections satisfy a containment, up to radical, in the opposite direction of the containment in \Cref{ideal_of_minors_containment}:

\begin{prop}
    Let $S$ be a regular local ring with infinite residue field, $f_1,\dots,f_c\in S$ a regular sequence, and $R=S/(f_1,\dots,f_c)$. We have $\on{rad}(\mathbf{F}^j_{i+2}(M))\subseteq \on{rad}(\mathbf{F}^j_{i}(M))$ for all $j\ge 1$ and $i\gg 0$, where $\on{rad}( - )$ denotes the radical.
\end{prop}
\begin{proof}
    We have a surjection $\syz_{i+2}(M)\onto \syz_i(M)$ for $i \gg 0$ by \cite[Proposition 1.1 and Theorem 3.1]{Eisenbud_1980}. These maps extend to surjections $\bigwedge^j \syz_{i+2}(M)\onto \bigwedge^j \syz_i(M)$ for all $j \ge 1$ and $i\gg 0$. By \cite[Corollary 1.3]{BE_77}, the ideals $\ann(\bigwedge^j \syz_i(M))$ and $\textbf{F}^j_{i+1}(M)$ agree up to radical for all $i \ge 0$ and $j \ge 1$. Thus, for $i\gg 0$ and $j \ge 1$, we have:
    $$
    \on{rad}(\textbf{F}^j_{i+2}(M))= \on{rad}\left(\ann(\bigwedge^j \syz_{i+1}(M)) \right) \subseteq \on{rad}\left(\ann(\bigwedge^j \syz_{i-1}(M))\right) = \on{rad}(\textbf{F}^j_{i}(M)).
    $$
\end{proof}

\section{Periodicity of ideals of minors over Golod rings}\label{golod_section}

Our goal in this section is to prove that minimal free resolutions of finitely generated modules over Golod rings have eventually 1- or 2-periodic ideals of minors. We prove this via an explicit calculation, relying on results of Burke \cite{Golod_burke} and Lescot \cite{lescot}.

We fix some notation for this section: let $S$ be a commutative ring, $R$ a cyclic $S$-algebra, and $M$ a finitely generated $R$-module. We choose $S$-free resolutions $(A, d^A) \xra{\simeq} R$ and $(G, d^G) \xra{\simeq} M$, and we assume $A_0 = S$. Let $A_+$ denote the brutal truncation $\bigoplus_{i \ge 1} A_i$.

\subsection{The bar resolution}

We now (roughly) recall the notions of $A_\infty$-algebra and $A_\infty$-module structures. A recollection of the explicit definitions won't be necessary for this paper; we refer the reader to \cite{Golod_burke} for the details.

\begin{defn}
\label{def:ainf}
An \emph{$A_{\infty}$ $S$-algebra structure} on $A$ is given by degree $-1$ $S$-linear maps
$$
m_n:A_+[-1]^{\otimes n}\rightarrow A_+[-1]
$$
for $n \ge 1$ such that $m_1 = d^{A[-1]}|_{A_+[-1]}$ and satisfying the sequence of relations in \cite[Definition 3.1]{Golod_burke}.\footnote{We point out that our shifting convention is different from that of \cite{Golod_burke}: we assume $C[1]_j$ = $C_{j+1}$, while Burke's convention is $C[1]_j = C_{j-1}$.} 
Assuming that $A$ is equipped with an $A_{\infty}$-algebra structure, an \emph{$A_{\infty}$ $A$-module structure} on $G$ is given by degree $-1$ $S$-linear maps
$$
m_n^G:A_+[-1]^{\otimes n-1}\otimes_S G\rightarrow G
$$
for $n \ge 1$ such that $m_1^G = d^G$ and satisfying the relations in \cite[Definition 3.3]{Golod_burke}.
\end{defn}

\begin{prop}[\cite{Golod_burke} Proposition 3.6]
\label{exists}
There exists an $A_\infty$ $S$-algebra structure on $A$ and an $A_\infty$ $A$-module structure on $G$.
\end{prop}

Let us fix an $A_\infty$ $S$-module structure on $A$ and an $A_\infty$ $A$-module structure on $G$. 

\begin{constr}[\cite{Golod_burke} Definition 3.12]
\label{bar_resolution}
Given an element $x_1\otimes\dots\otimes x_n \in A_+[-1]^{\otimes n}$, we denote it as $[x_1|\dots|x_n]$. The \emph{bar complex associated to $A$ and $G$} is the complex of free $R$-modules with underlying graded module $\bigoplus_{n \ge 0}  A_+[-1]^{\otimes n} \otimes_S G \otimes_S R$ and differential given, on the summand $A_+[-1]^{\otimes n} \otimes_S G \otimes_S R$, by the formula:
$$
d = \left( \sum_{i = 1}^n \sum_{j = 0}^{i-1} b_{i,j} \right) + \left( \sum_{i = 1}^{n+1} b_i' \right), \text{ where, setting $\underline{x}=(-1)^{|x|+1}x$ for $x \in A_+[-1]$, we have}\footnote{In \cite[Definition 3.12]{Golod_burke} the sum over the maps $b_i'$ ranges from 0 to $n$. This is an error; the correct limits on the summation are $1$ and $n+1$.}
$$
\begin{align*}
 b_{i,j}([x_1|\cdots|x_n]\otimes g \otimes r) &= [\underline{x}_1|\cdots|\underline{x}_j|m_i([x_{j+1}|\cdots|x_{j+i}])| x_{j+i+1} | \cdots|x_n]\otimes g \otimes r, \text{ and} & \\
b'_i([x_1|\cdots|x_n]\otimes g \otimes r) &= [\underline{x}_1|\cdots|\underline{x}_{n-i+1}]\otimes m_i^G([x_{n-i+2}|\cdots|x_n]\otimes g) \otimes r.
\end{align*}
\end{constr}

\begin{thm}[\cite{Golod_burke} Theorem 3.13]
\label{thm:resolution}
The bar construction associated to $A$ and $G$ is an $R$-free resolution of $M$.
\end{thm}

\subsection{Golod rings and modules}
\label{Golod_periodicity}

We recall here some background on Golod rings and modules; we refer the reader to \cite{avramov} for a comprehensive introduction to this topic. Let us assume now that $S$ is a regular local ring with maximal ideal $\m$ and residue field $k$. Let $\varphi : S \to R$ denote the canonical map. We assume that our $S$-free resolutions $A \xra{\simeq} R$ and $G \xra{\simeq} M$ are minimal. Given a local ring $T$ and a finitely generated $T$-module $N$, let $\beta_N^i$ denote the $i^{\th}$ Betti number of $N$ and $P_N^T(t)=\sum_i\beta^i_N t^i$ its Poincar\'e series.

\begin{dfn}
We say the $R$-module $M$ is \emph{$\varphi$-Golod} if we have an equality
\begin{equation}
\label{poincare}
P_M^R(t)=\dfrac{P_M^S(t)}{1-t(P_R^S(t)-1)}.
\end{equation}
The ring $R$ is \emph{$\varphi$-Golod} if $k$ is a $\varphi$-Golod $R$-module.
\end{dfn}

It is known that the notion of $\varphi$-Golodness is independent of the choice of presentation $\varphi$ \cite[\S 6]{Golod_burke}, and so we refer to $\varphi$-Golod rings and modules as simply Golod. The bar complex associated to choices of $A_\infty$-structures on $A$ and $G$ (Construction~\ref{bar_resolution}) has Betti numbers given by the coefficients of the power series on the right of \eqref{poincare}, as observed in \cite[Proof of Lemma 6.3]{Golod_burke}. Thus, we always have an inequality
$
P_M^R(t)\le\dfrac{P_M^S(t)}{1-t(P_R^S(t)-1)}
$
coefficient-wise, and one deduces:

\begin{lemma}[\cite{Golod_burke} Lemma 6.3]
\label{lem:golod}
Let $S$ be a regular local ring, $R$ a cyclic $S$-algebra, $M$ a finitely generated $R$-module, $A$ the minimal $S$-free resolution of $R$, and $G$ the minimal $S$-free resolution of $M$. A finitely generated $R$-module $M$ is Golod if and only if the bar construction associated to $A$ and $G$ is minimal with respect to some (equivalently every) $A_\infty$-structure on $A$ and $G$.
\end{lemma}

We also recall the following result of Lescot:

\begin{thm}[\cite{lescot}]\label{golod_syzygies}
\label{lescot}
Let $S$ be a regular local ring, $R$ a cyclic $S$-algebra that is Golod, and $M$ a finitely generated $R$-module. If $n > \dim(S)$, then the $n^{\th}$ syzygy of $M$ is Golod. 
\end{thm}

\subsection{Periodicity of ideals of minors of minimal free resolutions over Golod rings}

Our main result concerning ideals of minors over Golod rings is the following:

\begin{thm}\label{Golod_periodicity}
Let $S$ be a regular local ring, $R$ a cyclic $S$-algebra that is Golod, $M$ a finitely generated $R$-module, and $r > 0$. 
\begin{enumerate}
\item We have $\I^r_n(M) = \I^r_{n+2}(M)$ for $n\gg0$.
\item If $\pd_S(R) \ge 2$, we have $\I^r_n(M) = \I^r_{n+1 }(M)$ for $n\gg0$.
\end{enumerate}
\end{thm}

\begin{proof}
By Theorem~\ref{lescot}, we may assume $M$ is Golod. Let $A$ be the minimal $S$-free resolution of $R$ and $G$ the minimal $S$-free resolution of $M$. Applying Proposition~\ref{exists}, choose an $A_\infty$ $S$-algebra structure on $A$ and an $A_\infty$ $A$-module structure on $M$. By Theorem~\ref{thm:resolution} and Lemma~\ref{lem:golod}, the associated bar construction $(B, d)$ (Construction~\ref{bar_resolution}) is the minimal free resolution of $M$. 

If $A_t = 0$ for all $t > 0$, then $R = S$, and so $R$ is regular local. In this case, the statement of (1) is trivial, and (2) holds vacuously. We therefore may assume that $A_1 \ne 0$. 

We now proceed much like the proof of Proposition~\ref{prop:shamashc1}. For each $n \ge 0$, we have:
\begin{equation}
\label{barterm}
B_n = \bigoplus_{k \ge 0} \bigoplus_{\substack{\ell_1, \dots, \ell_k \ge 1, \text{ } m \ge 0, \\ \ell_1 + \cdots + \ell_k + m + k = n}} A_{\ell_1} \otimes_S \cdots \otimes_S A_{\ell_k} \otimes_S G_m \otimes_S R.
\end{equation}
Suppose $t > 0$ and $A_t \ne 0$. Fix $n \ge 0$. Notice that $A_t \otimes_S B_n$ is a summand of $B_{n + t + 1}$. Let $H_{n+t+1}$ denote the complement of $A_t \otimes_S B_n$ in the decomposition \eqref{barterm} for $B_{n+t+1}$. The bar differential $d_{n+t+1} : B_{n+t+1} \to B_{n+t}$ may be realized as a $2 \times 2$ matrix
\begin{equation}
\label{matrix}
d_{n+t+1} : (A_t \otimes_S B_n) \oplus H_{n+t+1} \to (A_t \otimes_S B_{n-1}) \oplus H_{n+t}.
\end{equation}
\vskip.1in
\noindent{\emph{Claim.}} The upper-left entry of the $2 \times 2$ matrix \eqref{matrix} is $\id_{A_t} \otimes d_n$. 
\vskip.1in
Before we prove the claim, let us explain how it quickly implies the desired results. Assuming the claim, we have $\I^r_n(M) \subseteq \I^r_{n+t+1}(M)$ for all $n \ge 0$. By the Noetherianity of $R$, we therefore have $\I^r_n(M) = \I^r_{n+t+1}(M)$ for $n \gg 0$; that is, the ideals of minors are eventually $(t +1)$-periodic. Notice that this holds for all $t > 0$ such that $A_t \ne 0$. Applying this with $t = 1$ proves (1). In the setting of (2), we have both $A_1 \ne 0$ and $A_2 \ne 0$, so the ideals of minors of $M$ are eventually both 2-periodic and 3-periodic, and hence 1-periodic.

Let us now prove the claim. We consider the action of the bar differential on a summand $Y$ of $B_{n+t+1}$ of the form $A_t \otimes_S A_{\ell_2} \otimes_S \cdots \otimes_S A_{\ell_k} \otimes_S G_m \otimes_S R$. 
Referring to the notation in Construction~\ref{bar_resolution}: the component $b_{i,j}$ (resp. $b'_i$) of the bar differential $d_{n + t + 1}$ acts on the summand $Y$ via the map $\id_{A_t} \otimes d_n$ whenever $j \ne 0$ (resp. when $i \ne k+1$). For $1 \le i \le n$, the component $b_{i, 0}$ of $d_{n+t + 1}$ maps $Y$ to $A_{t + \ell_2 + \cdots + \ell_{i} + i - 2} \otimes_S A_{\ell_i} \otimes_S \cdots \otimes_S A_{\ell_k} \otimes_S G_m \otimes_S R$; since $t + \ell_2 + \cdots + \ell_{i} + i - 2$ cannot be equal to $t$, the component $b_{i,0}$ does not contribute to the upper-left entry of the matrix \eqref{matrix}. Similarly, the component $b'_{k+1}$ maps $Y$ to $G_{n+t} \otimes_S R$, and so $b'_{k+1}$ also does not contribute to the upper-left entry of the matrix \eqref{matrix}. This proves the claim.
\end{proof}

\begin{example}
The assumption in Theorem~\ref{Golod_periodicity}(2) 
that $\pd_S(R) \ge 2$ is necessary. For instance, take $S = k[[ x ]]$ and $R = S/(x^3)$. The ring $R$ is Golod, but one easily checks that $\I^1_n(k) \ne \I^1_{n+1}(k)$ for all $n \ge 1$.
\end{example}

\begin{remark}
As in Theorem~\ref{CI_periodicity}, Theorem~\ref{Golod_periodicity} holds for $M$ a bounded below complex of finitely generated free $R$-modules with bounded homology, cf. Remark~\ref{ciremarks}.
\end{remark}

\section{Counterexample to periodicity of ideals of minors in minimal free resolutions by Gasharov-Peeva}
\label{sec:counterexample}
We have shown that ideals of minors of minimal free resolutions over local complete intersections and Golod local rings are eventually 2-periodic. Given $n \ge 1$, we now exhibit an Artinian Gorenstein local ring $R$ whose modules do not necessarily have $n$-periodic ideals of minors. We consider a ring studied by Gasharov-Peeva in \cite[Proposition 3.1]{Gasharov-Peeva_90}. 
Let $k$ be a field, $\alpha \in k$, $S=k[x_1,\dots,x_5]_{(x_1,\dots,x_5)}$, and $I \subseteq S$ the quadric ideal
$$
(x_1^2, x_2^2, x_3x_4, x_3x_5, x_4x_5, x_5^2, \alpha x_1x_3+x_2x_3,
x_1x_4+x_2x_4,x_3^2-x_2x_5+\alpha x_1x_5, x_4^2-x_2x_5+x_1x_5).
$$
The ring $R = S/I$ is clearly Artinian, and it is Gorenstein by \cite[Proposition 3.1]{Gasharov-Peeva_90}(i).

\begin{prop}
\label{prop:counter}
If the group of units of $k$ has an element of infinite order, then there is a finitely generated $R$-module $M$ such that 
\begin{enumerate}
\item $\I^1_i(M) \ne \I^1_{i+n}(M)$ for any $i \ge 0$ and $n \ge 1$, and 
\item over the ring $R/(x_5)$, we have $\I^1_i(M/x_5M) \ne \I^1_{i+n}(M/x_5M)$ for $i \ge 0$ and $n \ge 1$.
\end{enumerate}
\end{prop}

\begin{proof}
Let $\alpha$ be an element of the group of units of $k$ with infinite order. For $i \ge 0$, let $\delta_i$ denote the matrix
$
\begin{pmatrix}
x_1 & \alpha^i x_3+x_4 \\
0 & x_2  
\end{pmatrix}
$
with entries in $R$, and let $M = \im(\delta_0)$. By  \cite[Proposition 3.1]{Gasharov-Peeva_90}(ii), the complex 
$
\left[ R^2 \xla{\delta_1} R^2 \xla{\delta_2}  R^2 \xla{\delta_3} \cdots \right]
$
is the minimal free resolution of $M$. Observe that $\I_i^1(M) \ne \I_{i + n}^1(M)$ for $i \ge 0$ and $n \ge 1$. The last statement follows from \cite[Proposition 3.4]{Gasharov-Peeva_90}
\end{proof}

\begin{remark}
    Note that the example above is not a counterexample to the weaker version of periodicity suggested in Question \ref{introquestion1}.
\end{remark}

\section{Applications of the periodicity theorems}
\label{sec:moreresults}

In this final section, we apply our periodicity results (Theorems~\ref{CI_periodicity} and~\ref{Golod_periodicity}) to prove several additional facts about minimal free resolutions over complete intersections and Golod rings. We begin with a proof of Corollary~\ref{intro4}; let us restate the result here:

\begin{cor}
\label{cor:exterior}
    Let $S$ be a regular local ring, $I$ an ideal of $S$, and $M$ a finitely generated $R = S/I$-module. Let $\Syz_i(M)$ (resp. $\beta_i$) be the $i^{\th}$ syzygy (resp. Betti number) of $M$ over $R$. We have the following:
    \begin{enumerate}
        \item If $R$ is a complete intersection, then $\bigwedge^{\beta_{i}}\syz_i(M)\cong \bigwedge^{\beta_{i+2}}\syz_{i+2}(M)$ for $i \gg 0$. If, in addition, the minimal free resolution of $M$ is a Shamash resolution, then this isomorphism holds for $i \ge \dim(S) + 1$. 
        \item If $R$ is Golod, then $\bigwedge^{\beta_{i}}\syz_i(M)\cong \bigwedge^{\beta_{i+2}}\syz_{i+2}(M)$ for $i \gg 0$. If, in addition, we have $\pd_S(R) \ge 2$, then $\bigwedge^{\beta_{i}}\syz_i(M)\cong \bigwedge^{\beta_{i+1}}\syz_{i+1}(M)$ for $i \gg 0$.
        \item Suppose $\depth(R)=0$ and that the Burch index of $R$ (see \cite[Definition 1.2]{DE}) is at least two. For all $i\geq 5$, we have $\bigwedge^{\beta_i}\syz_i(M)\cong k$.
    \end{enumerate}
\end{cor}

\begin{proof}
    Let $N$ be a finitely generated $R$-module that is minimally generated by $\mu$ elements. Its top exterior power $\bigwedge^{\mu} N$ is cyclic and hence isomorphic to $R/\ann(\bigwedge^{\mu} N)$. By \cite[Corollary 1.3]{BE_77}, we have $\bigwedge^{\mu} N\cong R/\I^1(\varphi)$, where $\varphi$ is a minimal presentation matrix of $N$. Now apply \Cref{cor:eisenbud}, \Cref{CI_periodicity}, \Cref{Golod_periodicity}, and \cite[Theorem 0.1]{DE}.
\end{proof}

Next, we use Theorems~\ref{CI_periodicity} and~\ref{Golod_periodicity} to relate boundedness of Betti numbers over complete intersections and Golod rings to the vanishing of stable ideals of minors:

\begin{thm}
\label{thm:bounded}
Let $R$ be a local ring and $M$ a finitely generated $R$-module with rank. Denote by $\syz_i(M)$ (resp. $\beta_i$) the $i^{\th}$ syzygy (resp. Betti number) of $M$ over $R$. 
\begin{enumerate}
\item If $\itot(\syz_n(M),r) = 0$ for some $r > 0$ and $n \ge 0$, then there exists $N \gg 0$ such that $\beta_i < N$ for all $i \ge 0$.
\item Suppose $R$ is a complete intersection or Golod ring. We have $\istab(M,r)=0$ for some $r > 0$ if and only if there exists $N \gg 0$ such that $\beta_i < N$ for all $i \ge 0$.  
\end{enumerate} 
\end{thm}

\begin{proof}
Since $M$ has rank, the modules $\syz_i(M)$ have rank for all $i \ge 0$. Let $T$ be the total ring of fractions of $R$, and denote by $G_i$ the free $T$-module $\syz_i(M) \otimes_R T$ for $i \ge 0$. Let $(F,d)$ be the minimal free resolution of $M$. For each $i$, we have a split short exact sequence of $T$-modules
$0 \to G_{i+1} \to F_i \otimes_R T \to G_i \to 0$.
Notice that the rank of each $d_i$ (i.e. the size of the largest nonvanishing minor of $d_i$) is equal to 
the rank of $G_i$. We conclude that $\beta_i = \rk(d_i) + \rk(d_{i+1})$ for all $i > 0$. Now, suppose the Betti numbers of $M$ are unbounded. Choose $i > n$ such that $\beta_i \ge 2r - 1$. We have $\I^r_\ell(M) = 0$ for $\ell \ge n$ by assumption, and so $\rk(d_{\ell}) < r$ for all $\ell \ge n$. Thus, $\beta_i = \rk(d_i) + \rk(d_{i+1}) \le 2r-2$, a contradiction. This proves (1). Part (2) follows from (1), Theorem~\ref{CI_periodicity}, and Theorem~\ref{Golod_periodicity}. 
\end{proof}
 
We now use Theorems~\ref{CI_periodicity} and~\ref{Golod_periodicity} to establish some constraints on stable ideals of minors over complete intersections and Golod rings:

\begin{thm}\label{CI_singular_locus}
 Let $R$ be a local complete intersection or Golod ring and $M$ a finitely generated $R$-module with rank. Assume there exists $r \ge 1$ such that $\istab(M,r) \neq 0$.
 \begin{enumerate}
     \item We have $V(\istab(M,r)) \subseteq \{\p \in \spec(R) \text{ : } \pd_{R_{\p}}(M_{\p}) = \infty  \} \subseteq \sing(R)$. 
     \item If $R$ is a localization of a finitely generated algebra over a field or is equicharacteristic and excellent, then there is some $n_r \ge 1$ such that $\ca(R)^{n_r} \subseteq \istab(M,r)$.
    \item Assume $R$ is a complete intersection. If $R$ is a localization of a finitely generated algebra over a field and has a perfect residue field, then there is some $m_r \ge 1$ such that $\jac(R)^{m_r} \subseteq \istab(M,r)$. 
 \end{enumerate}
\end{thm}

\begin{remark}
In the setting of Theorem~\ref{CI_singular_locus}(3), it follows from \cite[Theorem 5.1]{Wang1994OnTF} that, when $\mathbb{Q} \subseteq R$, we have $\jac(\widehat{R})^r\subseteq \istab(\widehat{M},r)$, where $\widehat{R}$ denotes the completion of $R$.
\end{remark}
 
To prove Theorem~\ref{CI_singular_locus}, we will need the following well-known fact:

\begin{lemma} \label{singM}
    Let $R$ be a local ring, $M$ a finitely generated $R$-module with rank, and $(\F, d)$ the minimal free resolution of $M$. Set $\I_n(\F) \coloneqq \I^{\rk(d_n)}_n(\F)$.
    \begin{enumerate}
        \item For $\p \in \spec(R)$ and $n\geq 1$, $\pd_{R_{\p}}(M_{\p})> n$ if and only if $\p \in V(\I_n(\F))$.
        \item 
        For all $i \ge \depth(R)$, we have
        $
        V(\I_i(\F)) =  \{\p \in \spec(R) \text{ : } \pd_{R_{\p}}(M_{\p}) = \infty  \}.
        $
    \end{enumerate}
\end{lemma}

\begin{proof}
    By \cite[Proposition 1.4.9]{bruns_herzog} and \cite[Lemma 1.4.11]{bruns_herzog}, the $R_{\p}$-module $\coker(d_n)_{\p}$ is free if and only if $\I_n(\F)\not\subseteq \p$; this proves (1). (2) follows from (1) and the Auslander-Buchsbaum formula. 
\end{proof}

\begin{proof}[Proof of Theorem~\ref{CI_singular_locus}]
     Let $F$ be the minimal free resolution of $M$. Applying Theorems~\ref{CI_periodicity} and~\ref{Golod_periodicity}, choose an integer $i \ge \depth(R)$ such that $\istab(M,r)  = \I^r_i(\F)+\I^r_{i+1}(\F)$. By assumption, at least one of $\I^r_i(\F)$ and $\I^r_{i+1}(\F)$ is nonzero; without loss of generality, assume $\I^r_i(\F) \neq 0$. By \Cref{singM}(2), we have 
     $$
     V(\istab(M,r)) \subseteq V(\I^r_i(\F)) \subseteq V(\I_i(\F))  = \{\p \in \spec(R) \text{ : } \pd_{R_{\p}}(M_{\p}) = \infty  \} \subseteq \sing(R);
     $$
     where, as in Lemma~\ref{singM}, $\I_i(F) \coloneqq \I_i^{\rk(d_i)}(\F)$.
    This proves (1). (2) follows from (1) and \cite[Theorems 5.3 and 5.4]{iyengar_takahashi_ca}, and (3) follows from (1) and the Jacobian criterion.
 \end{proof}

\begin{cor}
Let $(R,\m)$ be a local complete intersection or Golod ring and $M$ a finitely generated $R$-module with rank.
\begin{enumerate} 
\item  If the Betti numbers of $M$ are unbounded, then the conclusions of \Cref{CI_singular_locus} hold for all $r \geq 1$.
    
    \item If $\pd_{R_{\p}} M_{\p} < \infty$ for all non-maximal primes $\p$ in $R$, then $\istab(M,r)$ is either 0 or $\m$-primary. In particular, if $R$ has an isolated singularity, then $\istab(M,r)$ is either 0 or $\m$-primary.
    \end{enumerate}
\end{cor}

\begin{proof}
(1) follows from Theorem~\ref{thm:bounded}(2), and (2) is immediate from \Cref{CI_singular_locus}(1).
\end{proof}

\bibliography{references}

\providecommand{\bysame}{\leavevmode\hbox to3em{\hrulefill}\thinspace}
\providecommand{\MR}{\relax\ifhmode\unskip\space\fi MR }
\providecommand{\MRhref}[2]{%
  \href{http://www.ams.org/mathscinet-getitem?mr=#1}{#2}
}
\providecommand{\href}[2]{#2}
\begin{thebibliography}{HWY16}

\bibitem[Avr86]{avramov}
Luchezar~L. Avramov, \emph{Golod homomorphisms}, Algebra, algebraic topology
  and their interactions ({S}tockholm, 1983), Lecture Notes in Math., vol.
  1183, Springer, Berlin, 1986, pp.~59--78. \MR{846439}

\bibitem[Avr10]{Avramov1998}
\bysame, \emph{Infinite free resolutions}, Six lectures on commutative algebra,
  Mod. Birkh\"{a}user Class., Birkh\"{a}user Verlag, Basel, 2010, pp.~1--118.
  \MR{2641236}

\bibitem[BE73]{BE}
David~A. Buchsbaum and David Eisenbud, \emph{What makes a complex exact?}, J.
  Algebra \textbf{25} (1973), 259--268. \MR{314819}

\bibitem[BE77]{BE_77}
\bysame, \emph{What annihilates a module?}, J. Algebra \textbf{47} (1977),
  no.~2, 231--243. \MR{476736}

\bibitem[BH93]{bruns_herzog}
Winfried Bruns and J\"{u}rgen Herzog, \emph{Cohen-{M}acaulay rings}, Cambridge
  Studies in Advanced Mathematics, vol.~39, Cambridge University Press,
  Cambridge, 1993. \MR{1251956}

\bibitem[BS15]{BS}
Jesse Burke and Greg Stevenson, \emph{The derived category of a graded
  {G}orenstein ring}, Commutative algebra and noncommutative algebraic
  geometry. {V}ol. {II}, Math. Sci. Res. Inst. Publ., vol.~68, Cambridge Univ.
  Press, New York, 2015, pp.~93--123. \MR{3496862}

\bibitem[Buc21]{buchweitz}
Ragnar-Olaf Buchweitz, \emph{Maximal {C}ohen-{M}acaulay modules and {T}ate
  cohomology}, Mathematical Surveys and Monographs, vol. 262, American
  Mathematical Society, Providence, RI, 2021, with appendices and an
  introduction by Luchezar L. Avramov, Benjamin Briggs, Srikanth B. Iyengar and
  Janina C. Letz. \MR{4390795}

\bibitem[Bur15]{Golod_burke}
Jesse Burke, \emph{Higher homotopies and {G}olod rings}, arXiv:1508.03782
  (2015).

\bibitem[BW15]{BW}
Jesse Burke and Mark~E. Walker, \emph{Matrix factorizations in higher
  codimension}, Trans. Amer. Math. Soc. \textbf{367} (2015), no.~5, 3323--3370.
  \MR{3314810}

\bibitem[Che43]{Chevalley}
Claude Chevalley, \emph{On the theory of local rings}, Annals of Mathematics
  \textbf{44} (1943), no.~4, 690--708.

\bibitem[DE23]{DE}
Hailong Dao and David Eisenbud, \emph{Burch index, summands of syzygies and
  linearity in resolutions}, Bull. Iranian Math. Soc. \textbf{49} (2023),
  no.~2, 10. \MR{4549775}

\bibitem[DKT20]{DKT}
Hailong Dao, Toshinori Kobayashi, and Ryo Takahashi, \emph{Burch ideals and
  {B}urch rings}, Algebra \& Number Theory \textbf{14} (2020), no.~8,
  2121--2150. \MR{4172703}

\bibitem[DMS23]{DMS}
Hailong Dao, Sarasij Maitra, and Prashanth Sridhar, \emph{On reflexive and
  {I}-{U}lrich modules over curve singularities}, Transactions of the American
  Mathematical Society, Series B \textbf{10} (2023), no.~12, 355--380.

\bibitem[EG94]{EG}
David Eisenbud and Mark~L. Green, \emph{Ideals of minors in free resolutions},
  Duke Math. J. \textbf{75} (1994), no.~2, 339--352. \MR{1290196}

\bibitem[EH05]{EH}
David Eisenbud and Craig Huneke, \emph{A finiteness property of infinite
  resolutions}, J. Pure Appl. Algebra \textbf{201} (2005), no.~1-3, 284--294.
  \MR{2158760}

\bibitem[Eis77]{eisenbudenriched}
David Eisenbud, \emph{Enriched free resolutions and change of rings},
  S\'{e}minaire d'{A}lg\`ebre {P}aul {D}ubreil, 29\`eme ann\'{e}e ({P}aris,
  1975--1976), Lecture Notes in Math., vol. 586, Springer, Berlin, 1977,
  pp.~1--8. \MR{0568883}

\bibitem[Eis80]{Eisenbud_1980}
\bysame, \emph{Homological algebra on a complete intersection, with an
  application to group representations}, Trans. Amer. Math. Soc. \textbf{260}
  (1980), no.~1, 35--64. \MR{570778}

\bibitem[EP16]{EP}
David Eisenbud and Irena Peeva, \emph{Minimal free resolutions over complete
  intersections}, Lecture Notes in Mathematics, vol. 2152, Springer, Cham,
  2016. \MR{3445368}

\bibitem[EP21]{EP_LR}
\bysame, \emph{Layered resolutions of {C}ohen-{M}acaulay modules}, J. Eur.
  Math. Soc. (JEMS) \textbf{23} (2021), no.~3, 845--867. \MR{4210725}

\bibitem[ERT94]{ERT}
David Eisenbud, Alyson Reeves, and Burt Totaro, \emph{Initial ideals,
  {V}eronese subrings, and rates of algebras}, Adv. Math. \textbf{109} (1994),
  no.~2, 168--187. \MR{1304751}

\bibitem[Far17]{farkas}
Gavril Farkas, \emph{Progress on syzygies of algebraic curves}, Moduli of
  curves, Lect. Notes Unione Mat. Ital., vol.~21, Springer, Cham, 2017,
  pp.~107--138. \MR{3729076}

\bibitem[GIK20]{GIK}
Shiro Goto, Ryotaro Isobe, and Shinya Kumashiro, \emph{Correspondence between
  trace ideals and birational extensions with application to the analysis of
  the {G}orenstein property of rings}, J. Pure Appl. Algebra \textbf{224}
  (2020), no.~2, 747--767. \MR{3987974}

\bibitem[GL69]{Gulliksen1969HomologyOL}
Tor~H. Gulliksen and Gerson Levin, \emph{Homology of local rings}, Queen's
  Papers in Pure and Applied Mathematics, No. 20, Queen's University, Kingston,
  Ont., 1969. \MR{0262227}

\bibitem[GP90]{Gasharov-Peeva_90}
Vesselin~N. Gasharov and Irena~V. Peeva, \emph{Boundedness versus periodicity
  over commutative local rings}, Trans. Amer. Math. Soc. \textbf{320} (1990),
  no.~2, 569--580. \MR{967311}

\bibitem[Gul68]{gulliksen}
Tor~Holtedahl Gulliksen, \emph{A proof of the existence of minimal
  {$R$}-algebra resolutions}, Acta Math. \textbf{120} (1968), 53--58.
  \MR{224607}

\bibitem[HH13]{HH}
J\"{u}rgen Herzog and Craig Huneke, \emph{Ordinary and symbolic powers are
  {G}olod}, Adv. Math. \textbf{246} (2013), 89--99. \MR{3091800}

\bibitem[HJ12]{HJ}
Raymond~C. Heitmann and David~A. Jorgensen, \emph{Are complete intersections
  complete intersections?}, J. Algebra \textbf{371} (2012), 276--299.
  \MR{2975397}

\bibitem[HS06]{HS}
Craig Huneke and Irena Swanson, \emph{Integral closure of ideals, rings, and
  modules}, London Mathematical Society Lecture Note Series, vol. 336,
  Cambridge University Press, Cambridge, 2006. \MR{2266432}

\bibitem[HWY16]{HWY}
J\"{u}rgen Herzog, Volkmar Welker, and Siamak Yassemi, \emph{Homology of powers
  of ideals: {A}rtin-{R}ees numbers of syzygies and the {G}olod property},
  Algebra Colloq. \textbf{23} (2016), no.~4, 689--700. \MR{3563560}

\bibitem[IT15]{iyengar_takahashi_ca}
Srikanth~B. Iyengar and Ryo Takahashi, \emph{{Annihilation of Cohomology and
  Strong Generation of Module Categories}}, International Mathematics Research
  Notices \textbf{2016} (2015), no.~2, 499--535.

\bibitem[IT21]{IYENGAR2021280}
\bysame, \emph{The {J}acobian ideal of a commutative ring and annihilators of
  cohomology}, J. Algebra \textbf{571} (2021), 280--296. \MR{4200721}

\bibitem[Iye97]{iyengarchange}
Srikanth Iyengar, \emph{Free resolutions and change of rings}, J. Algebra
  \textbf{190} (1997), no.~1, 195--213. \MR{1442152}

\bibitem[KL98]{KOH1998671}
Jee Koh and Kisuk Lee, \emph{Some restrictions on the maps in minimal
  resolutions}, J. Algebra \textbf{202} (1998), no.~2, 671--689. \MR{1617663}

\bibitem[Les90]{lescot}
Jack Lescot, \emph{S\'{e}ries de {P}oincar\'{e} et modules inertes}, J. Algebra
  \textbf{132} (1990), no.~1, 22--49. \MR{1060830}

\bibitem[Lin16]{Trace_Lindo}
Haydee M.~A. Lindo, \emph{Trace ideals and the centers of endomorphism rings of
  modules over commutative rings}, Ph.D. thesis, The University of Utah, 2016.

\bibitem[M2]{M2}
\emph{Macaulay2, a system for computation in algebraic geometry and commutative
  algebra programmed by {D}. {G}rayson and {M}. {S}tillman}.

\bibitem[Mar21]{martin}
A.~Amadeus Martin, \emph{Curved {BGG} correspondence}, Ph.D. thesis, The
  University of Nebraska-Lincoln, 2021.

\bibitem[MP15]{McCPeeva}
Jason McCullough and Irena Peeva, \emph{Infinite graded free resolutions},
  Commutative algebra and noncommutative algebraic geometry. {V}ol. {I}, Math.
  Sci. Res. Inst. Publ., vol.~67, Cambridge Univ. Press, New York, 2015,
  pp.~215--257. \MR{3525473}

\bibitem[Nag58]{Nag58}
Masayoshi Nagata, \emph{A general theory of algebraic geometry over {D}edekind
  domains. {II}. {S}eparably generated extensions and regular local rings},
  Amer. J. Math. \textbf{80} (1958), 382--420. \MR{94344}

\bibitem[Nag62]{Nagata}
\bysame, \emph{Local rings}, Interscience Tracts in Pure and Applied
  Mathematics, No. 13, Interscience Publishers a division of John Wiley \&
  Sons\, New York-London, 1962.

\bibitem[NR54]{NR}
D.~G. Northcott and D.~Rees, \emph{Reductions of ideals in local rings}, Proc.
  Cambridge Philos. Soc. \textbf{50} (1954), 145--158. \MR{59889}

\bibitem[Orl09]{orlov}
Dmitri~O. Orlov, \emph{Derived categories of coherent sheaves and triangulated
  categories of singularities}, Algebra, arithmetic, and geometry, Springer,
  2009, pp.~503--531.

\bibitem[PR90]{Popescu-Roczen}
Dorin Popescu and Marko Roczen, \emph{Indecomposable {C}ohen-{M}acaulay modules
  and irreducible maps}, vol.~76, 1990, Algebraic geometry (Berlin, 1988),
  pp.~277--294. \MR{1078867}

\bibitem[Rob98]{roberts_1998}
Paul~C. Roberts, \emph{Multiplicities and {C}hern classes in local algebra},
  Cambridge Tracts in Mathematics, vol. 133, Cambridge University Press,
  Cambridge, 1998. \MR{1686450}

\bibitem[Sha71]{SHAMASH19711}
Jack Shamash, \emph{The {P}oincar\'{e} series of a local ring. {II}}, J.
  Algebra \textbf{17} (1971), 1--18. \MR{269646}

\bibitem[Tat57]{Tate_57}
John Tate, \emph{Homology of {N}oetherian rings and local rings}, Illinois J.
  Math. \textbf{1} (1957), 14--27. \MR{86072}

\bibitem[Vas91]{Vasconcelos1991}
Wolmer~V. Vasconcelos, \emph{Computing the integral closure of an affine
  domain}, Proceedings of the American Mathematical Society \textbf{113}
  (1991), no.~3, 633--633.

\bibitem[Wan94]{Wang1994OnTF}
Hsin-Ju Wang, \emph{On the {F}itting ideals in free resolutions}, Michigan
  Math. J. \textbf{41} (1994), no.~3, 587--608. \MR{1297711}

\bibitem[Yos90]{yoshino}
Yuji Yoshino, \emph{Cohen-{M}acaulay modules over {C}ohen-{M}acaulay rings},
  London Mathematical Society Lecture Note Series, vol. 146, Cambridge
  University Press, Cambridge, 1990. \MR{1079937}

\end{thebibliography}
\bibliographystyle{amsalpha}
\Addresses

\end{document}